\documentclass[]{article}
\usepackage{mathrsfs,subfig,latexsym,array,multirow}
\usepackage{amsmath,amssymb,amsthm,datetime,graphicx}
\usepackage{color, amsbsy}

\title{
  A locally calculable $P^3$-pressure in a decoupled method  for 
  incompressible Stokes equations
}
  \author{
Chunjae Park
\thanks{Department of Mathematics,
Konkuk University, Seoul 05029, Korea.
\hspace{1mm}{\texttt{cjpark@konkuk.ac.kr}}}}

\DeclareMathAlphabet{\mathpzc}{OT1}{pzc}{m}{it}

\begin{document}

\newtheorem{theorem}{Theorem}[section]
\newtheorem{remark}[theorem]{Remark}
\newtheorem{lemma}[theorem]{Lemma}
\newtheorem{corol}[theorem]{Corollary}
\newtheorem{proposition}[theorem]{Proposition}
\newtheorem{definition}[theorem]{Definition}
\newtheorem{assumption}{Assumption}[section] 

\newcommand{\stg}[2]{\mathfrak{s}_{#1#2}}
\newcommand{\nsg}[2]{\mathcal{N}_{#1#2}}

\def\disp{\displaystyle}
\def\pskip{\hspace{1pt}}
\def\mmskip{\vspace{1mm}}

\def\R{\mathbb R}
\def\O{\Omega}
\def\p{\partial}
\def\Th{\mathcal{T}_h}

\def\div{\mathrm{div}\hspace{0.5mm}}
\def\curl{\mathbf{curl}\hspace{0.5mm}}

\def\alp{\alpha}
\def\bet{\beta}
\def\gam{\gamma}
\def\del{\delta}
\def\lam{\lambda}
\def\th{\theta}

\def\f{\mathbf{f}}
\def\u{\mathbf{u}}
\def\v{\mathbf{v}}
\def\w{\mathbf{w}}
\def\x{\mathbf{x}}
\def\y{\mathbf{y}}
\def\z{\mathbf{z}}
\def\n{\mathbf{n}}
\def\btau{\boldsymbol{\tau}}
\def\bnu{\boldsymbol{\nu}}
\def\bxi{\boldsymbol{\xi}}

\def\C{\mathbf{C}}
\def\M{\mathbf{M}}
\def\G{\mathbf{G}}
\def\P{\mathbf{P}}
\def\V{\mathbf{V}}
\def\W{\mathbf{W}}
\def\X{\mathbf{X}}
\def\Y{\mathbf{Y}}
\def\Z{\mathbf{Z}}

\newcommand{\Evec}[2]{\overrightarrow{#1#2}}
\newcommand{\PO}[1]{\mathcal{P}_h^{#1}(\Omega)}

\def\Xh{[\mathcal{P}_h^4(\O)\cap H_0^1(\O)]^2}
\def\Mh{\mathcal{P}_h^3(\O)}

\def\BB{\mathcal{B}}
\def\KK{\mathcal{K}}

\def\ST{\mathcal{S}}
\def\NS{\mathcal{N}}
\def\CS{\mathcal{C}}
\def\st{\mathfrak{s}}
\def\ns{\mathfrak{n}}
\def\cs{\mathfrak{c}}

\def\QQ{{Q}}
\def\m{\mathpzc{m}}
\def\CC{\mathfrak{C}}
\def\pihp{\Pi_h p}
\def\ph{p_h}
\def\JJ{\mathcal{J}}
\def\jump{\mathfrak{J}}
\def\jjp{j\hspace{1pt}j+1}
\def\tu{\widetilde{\u}}
\def\t{\boldsymbol{\tau}}

\def\RR{\mathpzc{R}}
\def\SS{\mathpzc{S}}

\maketitle
\begin{abstract}
   This paper  will suggest a new finite element method to find a $P^4$-velocity
  and a $P^3$-pressure solving incompressible Stokes equations at low cost.
  The method solves first the decoupled equation for a $P^4$-velocity.
  Then, using the calculated velocity,
  a locally calculable $P^3$-pressure will be defined component-wisely.
  The resulting $P^3$-pressure is analyzed to have the optimal order of
  convergence. 
  Since the pressure is calculated by local computation only, 
  the chief time cost of the new method is on solving the 
  decoupled equation for the $P^4$-velocity.
  Besides,  the method overcomes the problem of singular vertices or corners.
\end{abstract}

\section{Introduction}
High order finite element methods for incompressible Stokes equations
 have been developed well in 2 dimensional domain and analyzed along with the inf-sup condition
 \cite{Ainsworth1, Ainsworth2, Falk2013, Guzman2018, Scott1985}.
 They, however, endure their large degrees of freedom
 and have to avoid  singular vertices or corners.

 In the Scott-Vogelius finite element method, the inf-sup condition fails
 if the mesh has an exact singular vertex.
 Even on nearly singular vertices, the pressure solution is easy to be spoiled.
 Recently, to fix the problem, 
 we have found a cause of singular vertex
 and devised a new error analysis based on a so called sting function.
 As a result, the ruined pressure can be restored by simple post-process \cite{Park2020}.

 In this paper, employing the previous new error analysis,
 we will suggest  a new finite element method to find a $P^4$-velocity
  and a $P^3$-pressure solving incompressible Stokes equations at low cost.
 
 The method will solve first the decoupled equation for a divergence-free
 $P^4$-velocity which is almost same as the one from the Falk-Neilan finite element method
 except corners \cite{Falk2013}.
 Then, utilizing the calculated velocity and orthogonal decomposition of $P^3$,
 the 5 locally calculable components of a $P^3$-pressure will be defined
 by exploring locally calculable components
 in the Falk-Neilan and Scott-Vogelius finite element spaces
 \cite{Falk2013, Guzman2018, Scott1985}.
 The resulting $P^3$-pressure is analyzed to have the optimal order of convergence.

 Since the $P^3$-pressure is calculated by local computation only,
 the chief time cost of the new method is on solving the decoupled equation for the
 $P^4$-velocity.
 If the pressure has a region of interest in $\O$, the regional computation is enough for it.
 Besides, the method overcomes the problem arising from the singular
  vertices or corners by using the jump of the a priori calculated pressure components.
  
  In the overall paper, the characteristics of sting functions depicted
  in Figure \ref{fig:sting}-(a)
  play key roles as in the previous work in  \cite{Park2020}.
  Since the sting function exists in $P^k$ for every integer $k\ge0$,  
  the results for $P^4-P^3$ in this paper are easily extended for $P^{k+1}-P^k$, $k\ge 4$ 
  \cite{Park2021}.
  
The paper is organized as follows.
In the next section, the detail on finding a $P^4$-velocity  will be offered.
We will introduce an orthogonal decomposition of the space of $P^3$-pressures,
based on the orthogonality of sting and non-sting functions
in Section \ref{sec:orho-decom}.
Then, the most sections
will be devoted to defining the non-sting  component for each triangle
in Section \ref{sec:non-sting}
and the sting component for each vertex classified by
regular vertices, nearly singular ordinary vertices and dead corners
in Section \ref{sec:clust}-\ref{sec:dead}.
After the piecewise constant component is done in Section \ref{sec:const},
the final $P^3$-pressure will be defined by summing up all the components
in Section \ref{sec:def-p3-pressure}.
In the last two sections, a summary of the method and numerical tests will be given.

Throughout the paper, for a set $S \subset \mathbb R^2$,
standard notations for Sobolev spaces are employed and
$L_0^2(S)$ is the space of all $f\in L^2(S)$ whose integrals over $S$ vanish.
We will use $\|\cdot\|_{m,S}$, $|\cdot|_{m,S}$ and  $(\cdot,\cdot)_S$
for the norm, seminorm for $H^m(S)$ and  $L^2(S)$  inner product, respectively. 
If $S=\O$, it may be omitted in the subscript.
Denoting by $P^k$, the space of all polynomials of degree less than or equals $k$,
$f\big|_S\in P^k$ will mean that $f\in L^2(\O)$ coincides with a polynomial in $P^k$ on $S$.

\section{Velocity from the decoupled equation}
Let $\O$ be a simply connected polygonal domain in $\R^2$.
In this paper, we will approximate a pair of velocity and pressure
$(\u,p)\in [H_0^1(\O)]^2 \times L_0^2(\O)$ which satisfies
an incompressible Stokes equation:
\begin{equation}\label{prob:conti}
  ( \nabla \u,\nabla \v) +(p,\div \v)+(q,\div \u)
  = (\mathbf{f},\v) \quad \mbox{ for all } (\v,q) \in [H_0^1(\O)]^2\times  L_0^2(\O),
\end{equation}
for a body force $\mathbf{f}\in [L^2(\O)]^2$.

Given a family of shape-regular triangulations $\{\Th\}_{h>0}$  of $\O$,
define $\mathcal{P}_h^k(\O)$ as the following space of piecewise polynomials:
\[\mathcal{P}_h^k(\O) =\{ v_h \in L^2(\O)\ : \ v_h\big|_K \in P^k
  \mbox{ for all triangles } K\in\Th \}, \quad k \ge 0.\]
Let $\Sigma_{h,0}$ be a space of $\mathcal{C}^1$-Argyris  $P^5$ triangle elements
\cite{Brenner2002, Ciarlet} such that
\begin{equation}\label{def:Argyris}
  \Sigma_{h,0}=\mathcal{P}_h^5(\O) \cap H_{0}^2(\O),
\end{equation}
where
\[ H_{0}^2(\O)=\{ \phi\in H^2(\O)\ :\ \phi, \phi_x, \phi_y \in H_0^1(\O)\}.\]
The degrees of freedom of $\phi\in\Sigma_{h,0}$ are
$\phi_{xx}, \phi_{xy}, \phi_{yy}, \phi_x, \phi_y, \phi$ at interior vertices,
$\phi_{\bnu}$ at midpoints of interior edges and $\phi_{\n\n}$ at non-corner boundary vertices,
where $\bnu, \n$ are unit vectors normal to edges, $\p\O$, respectively.

Define a divergence-free space $\Z_{h,0}$ as
\begin{equation}\label{def:zh0}
  \Z_{h,0}=\{ (\phi_{h,y},-\phi_{h,x})\  \ :\ \phi_h\in \Sigma_{h,0}\hspace{1pt} \} \subset
  [\mathcal{P}_h^4(\O)\cap H_0^1(\O)]^2.
\end{equation}
We note that
\[\dim\Sigma_{h,0}=\dim \Z_{h,0} = 6 \V^{\mbox{in}} + E^{\mbox{in}} + \V^{\mbox{bdy}} -\V^{\mbox{cnr}},     \]
where $\V^{\mbox{in}},  E^{\mbox{in}}, \V^{\mbox{bdy}}$, and $\V^{\mbox{cnr}} $
are the numbers of interior vertices, interior edges, boundary vertices and corners, respectively  \cite{Falk2013}.

Then, we can solve  $\u_h\in \Z_{h,0}$ satisfying the following decoupled equation:
\begin{equation}\label{eq:dclv-SC}
  (\nabla \u_h, \nabla \v_h)= (\f,\v_h)\quad \mbox{ for all } \v_h\in \Z_{h,0}.
\end{equation}
\begin{theorem}\label{th:vel-error}
  Let $(\u,p)\in [H_0^1(\O)]^2 \times L_0^2(\O)$ and  $\u_h\in \Z_{h,0}$
  satisfy  \eqref{prob:conti}, \eqref{eq:dclv-SC}, respectively.
   Then we estimate
  \begin{equation}\label{eq:th:vel-error-0}
    |\u -\u_h|_1 \le Ch^4 |\u|_5,
  \end{equation}
  if $\u\in [H^5(\O)]^2$, where $C$ is a constant independent of $h$.
\end{theorem}
\begin{proof}
  Since $(\u,p)\in [H_0^1(\O)]^2 \times L_0^2(\O)$ satisfies \eqref{prob:conti}, we have
  $\div\u=0$. Thus,
  there exists a stream function $\phi\in H_0^2(\O)$ such that \cite{GR}
  \begin{equation}\label{eq:th-uphi}
    \u=(\phi_y, -\phi_x).
  \end{equation}
  Let $\Pi_h\phi\in \Sigma_{h,0}$ be a projection of $\phi$ such that
  the Hessians, gradients  and values of $\phi-\Pi_h\phi$ at vertices
  and its normal derivatives at midpoints of edges vanish.
  Then, if $\u\in[H^5(\O)]^2$, by Bramble-Hilbert lemma, we have 
  \begin{equation}\label{eq:th:vel-error-1p}
    |\phi-\Pi_h\phi|_2 \le C h^4 |\phi|_6.
  \end{equation}
  
  If we denote $\Pi_h\u=\left((\Pi_h\phi)_y, -(\Pi_h\phi)_x\right)\in \Z_{h,0}$,
  then from \eqref{eq:th-uphi} and \eqref{eq:th:vel-error-1p}, we estimate
  \begin{equation}\label{eq:th:vel-error-1}
    |\u-\Pi_h\u|_1 \le Ch^4 |\u|_5.
  \end{equation}
We note that $(p,\div\v_h)=0$ for all $\v_h\in \Z_{h,0}$. Thus, 
  from \eqref{prob:conti} and \eqref{eq:dclv-SC}, we deduce
\[ (\nabla\u-\nabla\u_h,\nabla\v_h)=0\quad\mbox{ for all } \v_h\in \Z_{h,0}.\]
It is written in the form:
\begin{equation}\label{eq:th:vel-error-2}
  (\nabla\Pi_h\u-\nabla\u_h,\nabla\v_h)=(\nabla\Pi_h\u -\nabla\u,\nabla\v_h)
  \quad\mbox{ for all } \v_h\in \Z_{h,0}.
\end{equation}
Then we can establish
\eqref{eq:th:vel-error-0} from \eqref{eq:th:vel-error-1} and \eqref{eq:th:vel-error-2}
with $\v_h=\Pi_h\u-\u_h\in \Z_{h,0}$.
\end{proof}
\section{Orthogonal decomposition of $P^3$-pressures}\label{sec:orho-decom}
For a triangle $K\in\Th$ and an integer $k\ge0$, define 
\[ P^k(K)=\{ q\in L^2(\O) :\ q\big|_K\in P^k,\ q=0  \mbox{ on } \O\setminus K  \}.\]
In the remaining of the paper, we will use the following notations:
\begin{itemize}
\item[] $C$ : a generic constant which does not depend on $h$ of $\Th$,
\item[]  $\KK(\V)$ : the union of all triangles in $\Th$ sharing a vertex $\V$,
\item[] $\bxi^\perp$ : the counterclockwise  $90^\circ$ rotation of a vector $\bxi$,
\item[] $|S|$: the area or length of a set $S$,
\item[]  $\m(f)$: the average of a function $f$ over $\O$.
\end{itemize}
  
We assume the following on $\Th$ to exclude pathological meshes.
\begin{assumption}\label{asm:Th}
    No triangle in $\Th$ has two corner points of $\p\O$.
\end{assumption}

\subsection{sting function}
Let $\V$ be a vertex of a triangle $K$. Then there exists a unique
function $\st_{\V K}\in P^3(K)$ satisfying the following
quadrature rule:
\begin{equation}
  \label{def:sting}
  \int_K \st_{\V K}\ q\ dxdy=\frac{|K|}{100} q(\V)\quad \mbox{ for all } q\in P^3,
\end{equation}
since the both sides of \eqref{def:sting} are linear functionals on $P^3$.
If $\widehat K$ is a reference triangle with vertices $(0,0),(1,0),(0,1)$ and $\widehat\V=(0,1)$, we have
\begin{equation}
  \label{eq:exp-sting}
  \st_{\widehat\V \widehat K}(x,y)=\frac{28}5 y^3-\frac{63}{10} y^2 +\frac95y-\frac1{10},
\end{equation}
as depicted in Figure \ref{fig:sting}-(a).
Given a vertex $\V$ of $K$, we note that
\begin{equation}\label{eq:affine-sting}
  \st_{\V K}= \st_{\widehat\V \widehat K}\circ F^{-1}\quad\mbox{ for an affine transformation }
  F:\widehat{K}\longrightarrow K.
\end{equation}
Thus, the values of $\st_{\V K}$ are inherited from those of $\st_{\widehat\V \widehat K}$ as
\begin{equation}
  \label{eq:value-sting}
  \st_{\V K}(\V)=1,\quad \st_{\V K}\Big|_E=-\frac1{10} \mbox { on the opposite edge }E \mbox{ of }\V. 
\end{equation}
If $q=\alp\st_{\V K}$ for a scalar $\alp$,
we call it a sting function of $\V$ on $K$, named after the shape
of its graph as in Figure \ref{fig:sting}-(a).

For a triangle $K$, define a subspace of $P^3(K)$ as
\begin{equation}
  \label{def:stk}
  \ST(K)=<\st_{\V_1 K},\st_{\V_2 K}, \st_{\V_3 K}>,
\end{equation}
where $\V_1, \V_2, \V_3$ are 3 vertices of $K$. 
From \eqref{eq:value-sting}, it is easy to prove that
\begin{equation}
  \label{eq:dim-sting}
  \dim\ST(K)=3.
\end{equation}
\begin{figure}[ht]
  \subfloat[a sting function $\st_{\V K}$ of $\V$ on $K$ ]
  {\includegraphics[width=0.54\linewidth]{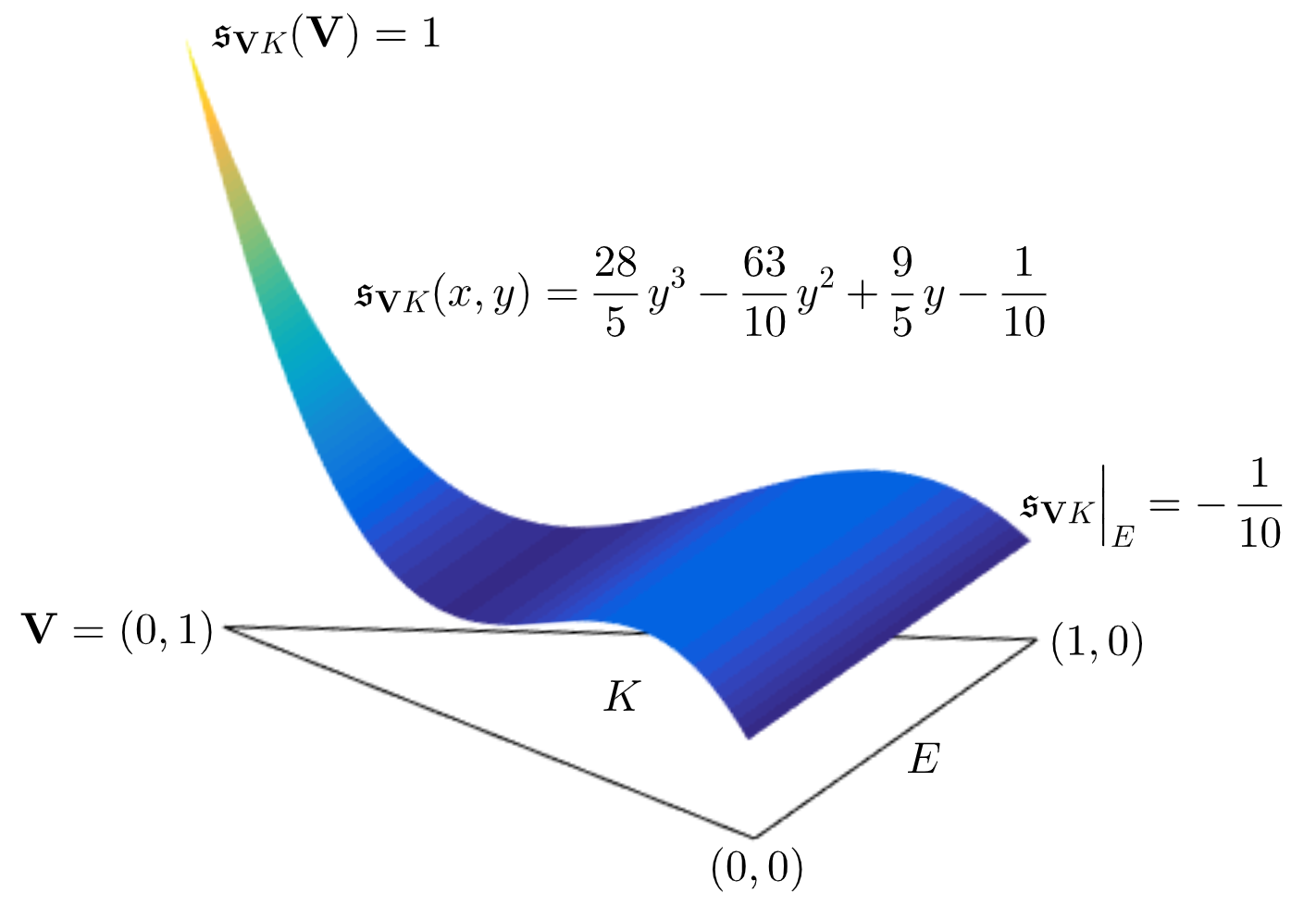}}
  \subfloat[a non-sting function $\ns$  on $K$]
  {\raisebox{2ex}{
      \includegraphics[width=0.45\linewidth]{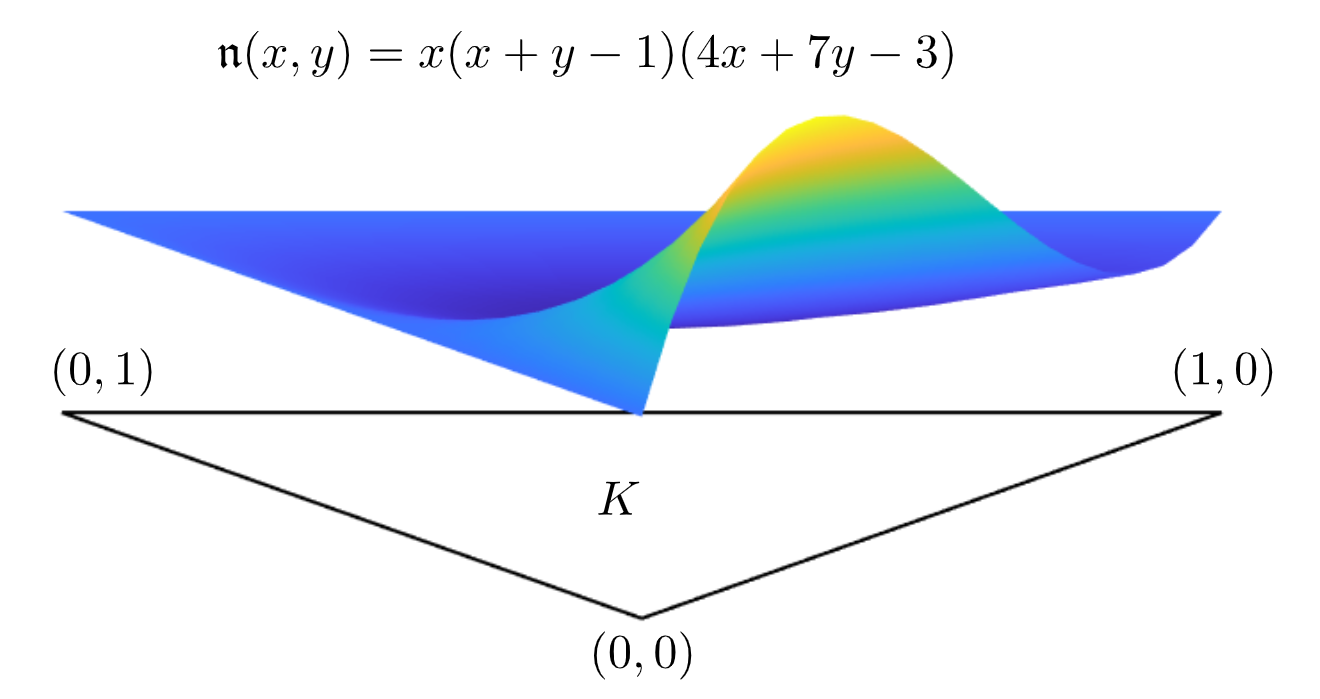}}
      }
\caption{examples of sting and non-sting functions  }
\label{fig:sting}
\end{figure}

\subsection{non-sting function}
For a triangle $K$, let
\begin{equation}\label{def:BK}
      \BB(K)=\{ \v\in [P^4(K)]^2\ : \ \v ={\bf{0}}\mbox{ on } \p K\},
    \end{equation}
and define a subspace of $P^3(K)$ as
\begin{equation}
  \label{def:nonsting}
  \NS(K)=\{\div\v\ :\ \v\in \BB(K)\}.
\end{equation}
If  $q\in \NS(K)$, we will call it a non-sting function on $K$.
By definition in \eqref{def:BK}, \eqref{def:nonsting},  every non-sting function $q\in\NS(K)$
has the following properties:
\begin{equation}
  \label{eq:prop-tb}
  q(\V)=0 \mbox{ for every vertex }\V\mbox{ of }K,\quad \int_K q\ dxdy=0.
 \end{equation}
An example of its graph is depicted in Figure \ref{fig:sting}-(b).

Then, the following orthogonality is clear from  \eqref{def:stk}, \eqref{eq:prop-tb} and the quadrature rule in \eqref{def:sting},
\begin{equation}\label{eq:st-ns-perp}
 \NS(K) \perp \ST(K).
\end{equation}
The fact $\dim \BB(K)=6$ induces the following,
with an aid of Lemma \ref{lem:bk-div0} below,
\begin{equation}\label{eq:dim-nsting}
  \dim\NS(K)=6.
\end{equation}
\begin{lemma}\label{lem:bk-div0}
  If $\v\in \BB(K)$ and $\div\v=0$, then $\v=\bf{0}$.
\end{lemma}
\begin{proof}
  Since $\div\v=0$ on $K$ and $\v=\bf{0}$ on $\p K$, there exists $\phi\in P^5$ such that
  \[ (\phi_y, -\phi_x)=\v \mbox{ on } K,\quad \phi=0\mbox{ on }\p K. \]
  Let $\ell_1, \ell_2, \ell_3$ be 3 infinite lines containing 3 line segments of $\p K$,
  respectively.
  Then $\phi, \nabla\phi$ vanish on $\ell_1, \ell_2, \ell_3$.
  It implies that $\phi$ vanishes on any line which passes 3 points
  in $\ell_1 \cup \ell_2 \cup \ell_3$.
  Thus we have $\phi=0$ and  $\v=\bf{0}$ on $K$.
\end{proof}
The above lemma tells that $\|\div\v\|_{0,K}$ is a norm of $\v\in \BB(K)$.
Furthermore, we can show that  \cite{Guzman2018}
\begin{equation}
  \label{eq:div-bk-norm}
   |\v|_{1,K} \le C \|\div\v\|_{0,K}\quad \mbox{ for all } \v\in \BB(K).
 \end{equation}
 Actually, $\NS(K)$ in \eqref{def:nonsting} is the space of all function $q\in P^3(K)$ satisfying  \eqref{eq:prop-tb}.

\subsection{orthogonal decomposition}
Let $1^K\in P^3(K)$ be a constant function of value 1 on $K$.
Then, we can decompose $P^3(K)$ as in the following lemma.
We will notate
\[ A\bigoplus^\perp B\quad\mbox{ for }A\bigoplus B, \mbox{ if } A\perp B. \]
\begin{lemma}\label{lem:p3kdecom}
  \begin{equation}
    \label{eq:lem-p3k-decom}
    P^3(K)=\NS(K) \bigoplus^\perp \left( \ST(K)\bigoplus  <1^K> \right).
  \end{equation}
\end{lemma}
\begin{proof}
From \eqref{eq:prop-tb}, we have $<1^K>\perp \NS(K)$. Thus, by \eqref{eq:dim-sting}, \eqref{eq:st-ns-perp},
  \eqref{eq:dim-nsting},
  it is enough to prove  \eqref{eq:lem-p3k-decom} that
  \[ 1^K \notin \ST(K). \]

  For the  vertices $\V_1, \V_2, \V_3$ of $K$, let
\begin{equation*}
  \st_1^K=\frac54(\st_{\V_1K} + \st_{\V_2K} + \st_{\V_3K}) \in \ST(K).
\end{equation*}
Then, by \eqref{eq:value-sting},
$ 1^K-\st_1^K$ vanishes at all vertices of $K$.
It means, from  \eqref{def:sting}, that
\begin{equation}\label{eq:lem-p3kdecom7}
  \int_K q(1^K-\st_1^K)\ dxdy=0 \quad\mbox{ for all } q\in\ST(K).
\end{equation}
Assume $1^K\in\ST(K)$. Then, $1^K- \st_1^K\in\ST(K)$ and  $1^K- \st_1^K=0$
from \eqref{eq:lem-p3kdecom7}.
It contradicts to
\begin{equation*}
   \int_K \st_{1}^K\ dxdy =\frac54\int_K \st_{\V_1K}+ \st_{\V_2K}+ \st_{\V_3K}\ dxdy
  =\frac3{80}|K|\neq |K|= \int_K 1^K\ dxdy.
\end{equation*}
\end{proof}

Let's define the following subspaces of $\PO3=\disp\bigoplus_{K\in\Th}P^3(K)$:
\begin{equation}\label{def:STNSCS}
 \NS_h=\bigoplus_{K\in\Th}\NS(K),\qquad
 \ST_h = \bigoplus_{K\in\Th} \ST(K),\qquad \CS_h=\bigoplus_{K\in\Th} <1^K>.
\end{equation}
Then by Lemma \ref{lem:p3kdecom}, we have
\begin{equation}
  \label{eq:decom-p3omega}
  \PO3 =  \NS_h \bigoplus^\perp\left(\ST_h \bigoplus  \CS_h \right).
\end{equation}
\subsection{decomposition of $\pihp$}
For $(\u,p)$
satisfying  \eqref{prob:conti}, let
$\pihp\in \Mh$ be a Hermite interpolation of $p$
such that
\begin{equation}\label{def:php}
  \nabla\pihp (\V)=\nabla p(\V),\quad  \pihp (\V)=p(\V),\quad \pihp(\G)=p(\G),
\end{equation}
at all vertices $\V$ and gravity centers $\G$ of triangles in $\Th$.
Then, if $p\in H^4(\O)$, we have
\begin{equation}
  \label{eq:pihp-inter-error}
  \|p-\pihp\|_0 \le Ch^4 |p|_4.
\end{equation}

By  \eqref{eq:decom-p3omega}, we can decompose $\pihp$ into
\begin{equation}\label{eq:decom-pihp}
\Pi_h p=   \pihp^{\NS} +\pihp^\ST  + \pihp^\CS,
\end{equation}
for
\begin{equation} \label{def:decom-pihp}
\pihp^{\NS}\in \NS_h, \quad \pihp^\ST \in \ST_h, \quad \pihp^{\CS}\in \CS_h,
\end{equation}
called the non-sting, sting and piecewise constant components of $\pihp$, respectively.
We will approximate them component-wisely, exploiting  
the following equation for $\pihp$ from \eqref{prob:conti}:
\begin{equation}\label{eq:pihp-0}
  (\Pi_hp,\div \v)=(\f,\v)-(\nabla\u,\nabla\v)-(p-\Pi_hp,\div\v)\quad \mbox{ for all }
  \v\in \left[H_0^1(\O)\right]^2.
\end{equation}

\section{Non-sting component for a triangle}\label{sec:non-sting}
Fix a triangle $K\in\Th$ and define an operator $L:\NS(K)\longrightarrow \BB(K)'$ so that,
if  $q\in\NS(K)$,
\[  L q(\v_h)=(q,\div\v_h)\mbox{ for all } \v_h\in\BB(K).\]
Then, $L$ is an isomorphism
from  the definition of $\NS(K)$ in  \eqref{def:nonsting} and
\[\dim\NS(K)=\dim\BB(K).\]

Thus, for each triangle $K\in\Th$, there exists a unique $\ph^K\in\NS(K)$ such that
\begin{equation}
  \label{cond:phK}
   (\ph^K,\div\v_h)=(\f,\v_h)-(\nabla\u_h,\nabla\v_h)\quad\mbox{ for all }
    \v_h\in \BB(K).
  \end{equation}

\begin{lemma}\label{lem:phns-error}
  Define
  \begin{equation*}
    p_h^\NS = \sum_{K\in\Th} p_h^K.
  \end{equation*}
  Then for $\pihp^\NS$ in \eqref{eq:decom-pihp}, we estimate
  \begin{equation}\label{eq:lem-phns}
    \|\pihp^\NS - \ph^\NS\|_{0,K} \le C(|\u-\u_h|_{1,K} + \|p-\pihp\|_{0,K})
    \quad\mbox{ for each } K\in\Th.
  \end{equation}
\end{lemma}
\begin{proof}
For each triangle $K\in\Th$, 
we note that  $\pihp^\NS$  in \eqref{eq:decom-pihp} satisfies that
   \begin{equation}
  \label{cond:pihptv-main}
  (\pihp^\NS,\div\v_h)=(\f,\v_h)-(\nabla\u,\nabla\v_h)-(p-\pihp,\div\v_h)
  \quad\mbox{ for all } \v_h\in \BB(K), 
\end{equation}
from  \eqref{eq:pihp-0} and orthogonality in Lemma \ref{lem:p3kdecom}.

Denote $e_h^K=\pihp^\NS\big|_K-\ph^\NS\big|_K$. Then
from   \eqref{cond:phK} and \eqref{cond:pihptv-main}, we have
  \begin{equation}
  \label{eq:error-rhk}
  (e_h^K,\div\v_h)=-(\nabla\u-\nabla\u_h,\nabla\v_h)-(p-\pihp,\div\v_h)
  \quad\mbox{ for all } \v_h\in \BB(K). 
\end{equation}
Since $e_h^K\in\NS(K)$,
the estimation \eqref{eq:lem-phns} comes from  \eqref{eq:div-bk-norm}, \eqref{eq:error-rhk}
and the definition of $\NS(K)$ in \eqref{def:nonsting}.
\end{proof}

\section{Clustering sting functions by vertex}\label{sec:clust}
\subsection{regular and nearly singular vertices}\label{sec:def-sing}
\def\vts{\vartheta_{\sigma}}
A vertex $\V$ is called exactly singular if the union of all edges sharing $\V$
is contained in the union of two infinite lines.
To be precise, let $K_1, K_2,\cdots,K_J$ be all triangles sharing $\V$
and denote by $\theta(K_j)$, the angle of $K_j$ at $\V$, $j=1,2,\cdots,J$. 
Define
\[\Upsilon(\V)=\{\theta(K_i) + \theta(K_j)\ : \ K_i\cap K_j \mbox{ is an edge, } i,j=1,2,\cdots,J \}. \]
Then $\V$ is called exactly singular  if and only if 
$\Upsilon(\V)=\{\pi\} \mbox{ or } \emptyset$.

For the quantitative definition of  nearly singularity, fix $\vartheta$ such that
\[ 0<\vartheta \le \inf\{ \theta\ :\  \theta \mbox{ is an angle of a triangle } K\in\Th,
  h>0  \}.  \]
Then  call a vertex $\V$ to be nearly singular if 
\begin{equation}\label{def:nsing}
    |\Theta -\pi| < \vartheta \quad \mbox{ for all } \Theta\in\Upsilon(\V),  
\end{equation}
otherwise regular.
We note that nearly singular vertices are isolated from each others
in the sense of the following lemma \cite{Park2020}.
\begin{lemma}\label{lem:isol-vtx}
  There is no interior edge connecting two nearly singular vertices.
\end{lemma}

\subsection{clustering sting functions by vertex}
For each vertex $\V$, let $\ST(\V)$ be the space of all sting
functions of $\V$, that is,
\begin{equation}\label{def:STV}
  \ST(\V)= <\st_{\V K_1},\st_{\V K_2},\cdots, \st_{\V K_J}>,
\end{equation}
where $K_1, K_2, \cdots, K_J$ are all triangles in $\Th$ sharing $\V$.
An example of a function in $\ST(\V)$ is represented in Figure \ref{fig:phv_alpha}.
The support of a function in $\ST(\V)$ belongs to $\KK(\V)$.
\begin{figure}[ht]
\centering
\includegraphics[width=0.4\linewidth]{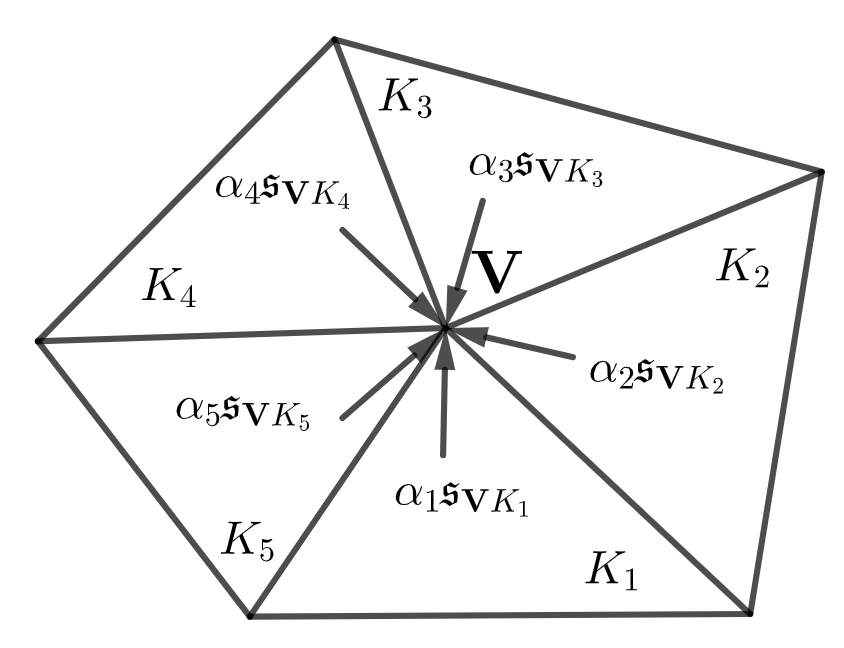}
\caption{the support of
  $q_h^\V=\alp_1\st_{\V K_1}+\alp_2\st_{\V K_2}+\cdots+ \alp_5\st_{\V K_5}\in\ST(\V)$}
\label{fig:phv_alpha}
\end{figure}

Then we note that 
\[ \ST_h =\bigoplus_{K\in\Th} \ST(K)= \bigoplus_{\V: \mbox{vertex}} \ST(\V).       \]
Thus the sting functions forming $\pihp^\ST$ of $\pihp$ in \eqref{eq:decom-pihp},
\eqref{def:decom-pihp}
can be clustered by vertex so that
\begin{equation}\label{eq:decom-pihp-st}
  \pihp^\ST= \sum_{\V: \mbox{vertex}} \pihp^\V \quad \mbox{ for } \pihp^\V\in\ST(\V).  
\end{equation}

If a vertex $\V$ does not meet any interior edge as in Figure \ref{fig:dead},
$\V$ is called dead, otherwise, ordinary.
Then, all vertices are classified into 3 classes: 
regular vertices, nearly singular ordinary vertices and dead corners
as in Figure \ref{fig:reg}, \ref{fig:sing}, \ref{fig:dead}, respectively.

In the next 3 sections, we will define the sting component $\ph^\V\in\ST(\V)$
for each vertex $\V$ in order of those 3 classes
to approximate $\pihp^\V$ in \eqref{eq:decom-pihp-st}.
The local functions in the following subsection will play roles of test functions
on defining $\ph^\V$.
  
\subsection{test functions on  two adjacent triangles}
Let $K_1, K_2$ be two adjacent triangles sharing an edge and a vertex $\V$
as in Figure \ref{fig:uniontwo}. Denote  other 3 vertices and a unit
tangent vector by $\W_0,\W_1,\W_2, \t$ so that
\[ \overline{\V\W_0}= K_1\cap K_2,\quad 
  \t = \frac{\Evec{\V}{\W_0}}{|\overline{\V\W_0}|},\quad
  \W_j\in K_j\setminus\{\V,\W_0\}, j=1,2. \]
\begin{figure}[ht]
  \centering
    \includegraphics[width=0.45\textwidth]{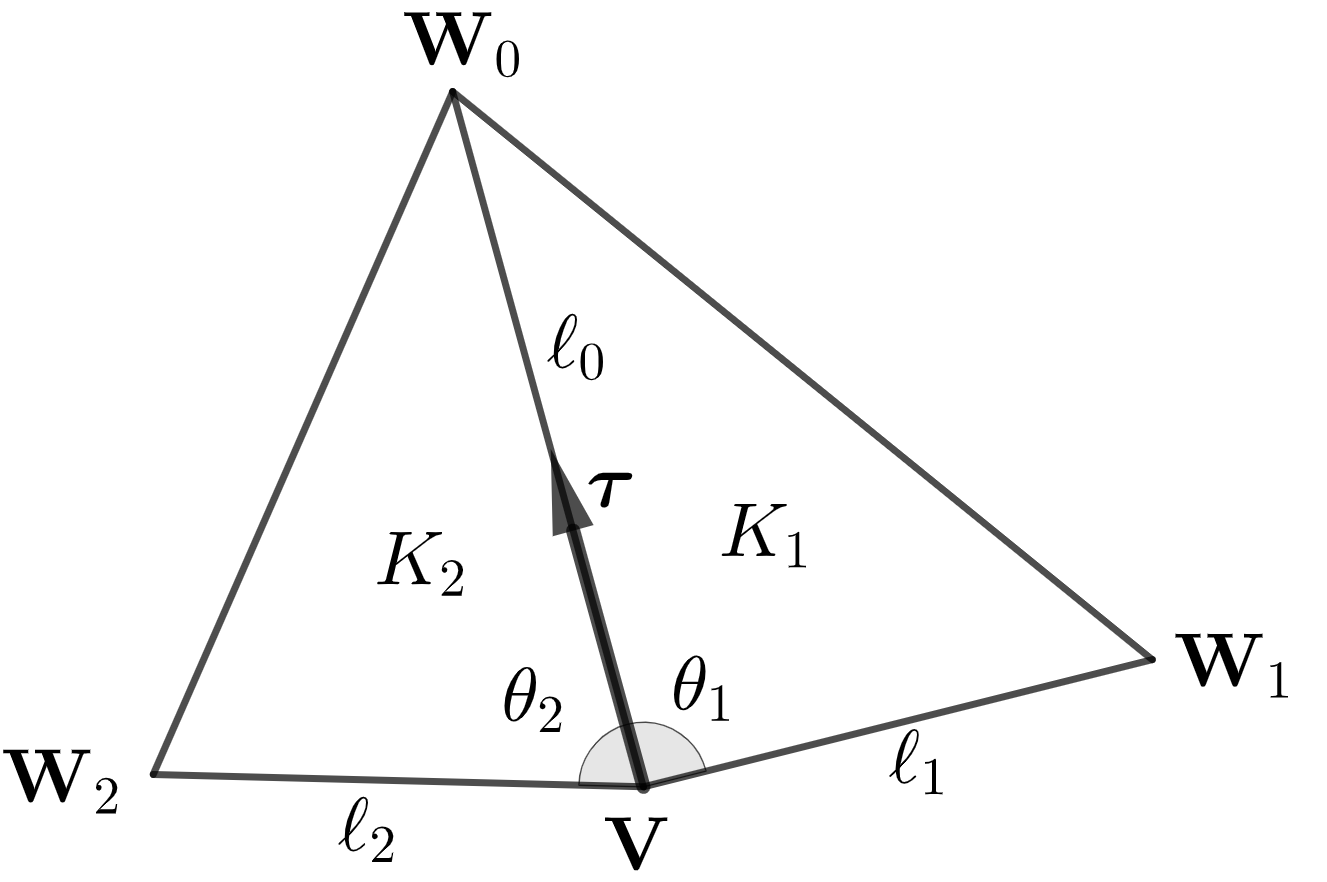} 
  \caption{The union of two adjacent triangles sharing $\V$}
  \label{fig:uniontwo}
\end{figure}

Then, there exists a function $w\in \mathcal{P}_h^4(\O)\cap H_0^1(\O)$
such that \cite{Park2020}
\begin{equation}\label{cond:w}
  \frac{\p w}{\p\t}(\V)=1,\ \frac{\p w}{\p\t} (\W_0)=0,\ \int_{\overline{\V\W_0}} w\ d\ell=0,\
   \mbox{ the support of  } w \mbox{ is } K_1\cup K_2.
\end{equation}
Assuming $K_1,K_2$ are counterclockwisely numbered with respect to $\V$,
by simple calculation, we have
  \begin{equation}\label{eq:nablaw12}
    \nabla w\big|_{K_1}(\V) =\frac{|\overline{\V\W_0}|}{2|K_1|}\
    {\Evec{\V}{\W_1}}^{\perp},\quad
    \nabla w\big|_{K_2}(\V) =-\frac{|\overline{\V\W_0}|}{2|K_2|}\
    {\Evec{\V}{\W_2}}^{\perp}.
  \end{equation}
  For a vector $\bxi=(\xi_1,\xi_2)$, denote  $\w_h^{\bxi}=w\bxi=(\xi_1w, \xi_2w)$.
  Then from \eqref{cond:w} and \eqref{eq:nablaw12}, $\div \w_h^{\bxi}$
  vanishes at all vertices in $\Th$ except
  \begin{equation}\label{eq:divwhval}
    \div\w_h^{\bxi}\big|_{K_1} (\V)=\frac{|\overline{\V\W_0}|}{2|K_1|}\ {\Evec{\V}{\W_1}}^{\perp}\cdot \bxi,\quad
    \div\w_h^{\bxi}\big|_{K_2} (\V)=-\frac{|\overline{\V\W_0}|}{2|K_2|}\ {\Evec{\V}{\W_2}}^{\perp}\cdot \bxi.
  \end{equation}

  Let $q_h \in \ST(K_1)\bigoplus\ST(K_2)$ be a sting function on $K_1\cup K_2$.
  It is
  represented with some constants $\alp_j, \bet_j,j=1,2,3$ as
  \[ q_h =\alp_1\stg{\V}{K_1} + \alp_2\stg{\W_0}{K_1}+
    \alp_3\stg{\W_1}{K_1} +\bet_1\stg{\V}{K_2} +
    \bet_2\stg{\W_0}{K_2}+ \bet_3\stg{\W_2}{K_2}.\]
  Then, by  \eqref{eq:divwhval} and quadrature rule of sting functions in \eqref{def:sting},
   we have that
  \begin{equation*}
     \resizebox{1\hsize}{!}
     {$
       \arraycolsep=1.0pt\def\arraystretch{2.5}
    \begin{array}{lll}
    (q_h, \div\w_h^{\t})
    &=&\alp_1\div\w_h^{\t}\big|_{K_1}(\V)\disp\frac{|K_1|}{100}+\bet_1\div\w_h^{\t}\big|_{K_2}(\V)\frac{|K_2|}{100}
    = \disp\frac{1}{200}\left(\alp_1{\Evec{\V}{\W_1}}^{\perp}
      -\bet_1{\Evec{\V}{\W_2}}^{\perp}\right) \cdot \Evec{\V}{\W_0},\\
     (q_h, \div\w_h^{\t^\perp})
      &=&\alp_1\div\w_h^{\t^\perp}\big|_{K_1}(\V)\disp\frac{|K_1|}{100}
          +\bet_1\div\w_h^{\t^\perp}\big|_{K_2}(\V)\disp\frac{|K_2|}{100}
     = \disp\frac{1}{200}
    \left(\alp_1{\Evec{\V}{\W_1}}^{\perp}-\bet_1{\Evec{\V}{\W_2}}^{\perp}\right) \cdot{\Evec{\V}{\W_0}}^\perp.
    \end{array}
    $}
  \end{equation*}
  It can be written in simpler form:
  \begin{equation}\label{eq:qhdivwh0}
    \arraycolsep=1.4pt\def\arraystretch{2.0}
    \begin{array}{lll}
      (q_h, \div\w_h^{\t}) &=&  \disp\frac{\ell_1\ell_0\sin\theta_1}{200}\alp_1
                               +\frac{\ell_0\ell_2\sin\theta_2}{200}\bet_1
                               = \frac1{100}\left(|K_1|\alp_1 +|K_2|\bet_1\right),\\
      (q_h, \div\w_h^{\t^\perp})&=& \disp\frac{\ell_1\ell_0\cos\theta_1}{200}\alp_1
                                    -\frac{\ell_0\ell_2\cos\theta_2}{200}\bet_1,
    \end{array}
  \end{equation}
where $\theta_j$ is the angle of $K_j$ at $\V$, $j=1,2$ and
$\ell_j=|\overline{\V\W_j}|, j=0,1,2$ as in Figure \ref{fig:uniontwo}.

\section{$\ph^\V$ for a regular vertex $\V$}\label{sec:sting-regular}
Let's fix a vertex $\V$ and $K_1, K_2, \cdots, K_J$ be all triangles in $\Th$
sharing $\V$,  counterclockwisely numbered as in Figure \ref{fig:reg}, \ref{fig:sing}.
Denote by $\JJ$, the number of interior edges which meet $\V$, that is,
\begin{equation} \label{def:JJ}
  \JJ=\left\{
    \begin{array}{ll}
      J,&\quad \mbox{ if } \V \mbox{ is an interior vertex, }\\
      J-1,&\quad \mbox{ if } \V \mbox{ is a boundary vertex. }
    \end{array}\right.                      
\end{equation}
We will use the indices modulo $J$, if $\V$ is an interior vertex.
Then, for $j=1,2,\cdots, \JJ$,
let $\V_j$ and  $\t_j$ be a vertex and a unit vector, respectively, such that
\begin{equation}\label{def:Vj}
  \overline{\V\V_j}=K_j\cap K_{j+1},\quad \t_j = \frac{\Evec{\V}{\V_j}}{|\overline{\V\V_j}|}.
\end{equation}
In case of boundary vertex $\V$ as in Figure \ref{fig:reg}-(b) and \ref{fig:sing}-(a),(b),
denote by $\V_0, \V_J$, the vertices such that
\begin{equation} \label{def:v0vJ}
  \V_0\in K_1\setminus\{\V,\V_1\},\quad  \V_J\in K_J\setminus\{\V,\V_{J-1}\}.
\end{equation}
\begin{figure}[ht]
  \centering
  \subfloat[interior regular $\V$, $J=\JJ=5$]{
    \includegraphics[width=0.45\textwidth]{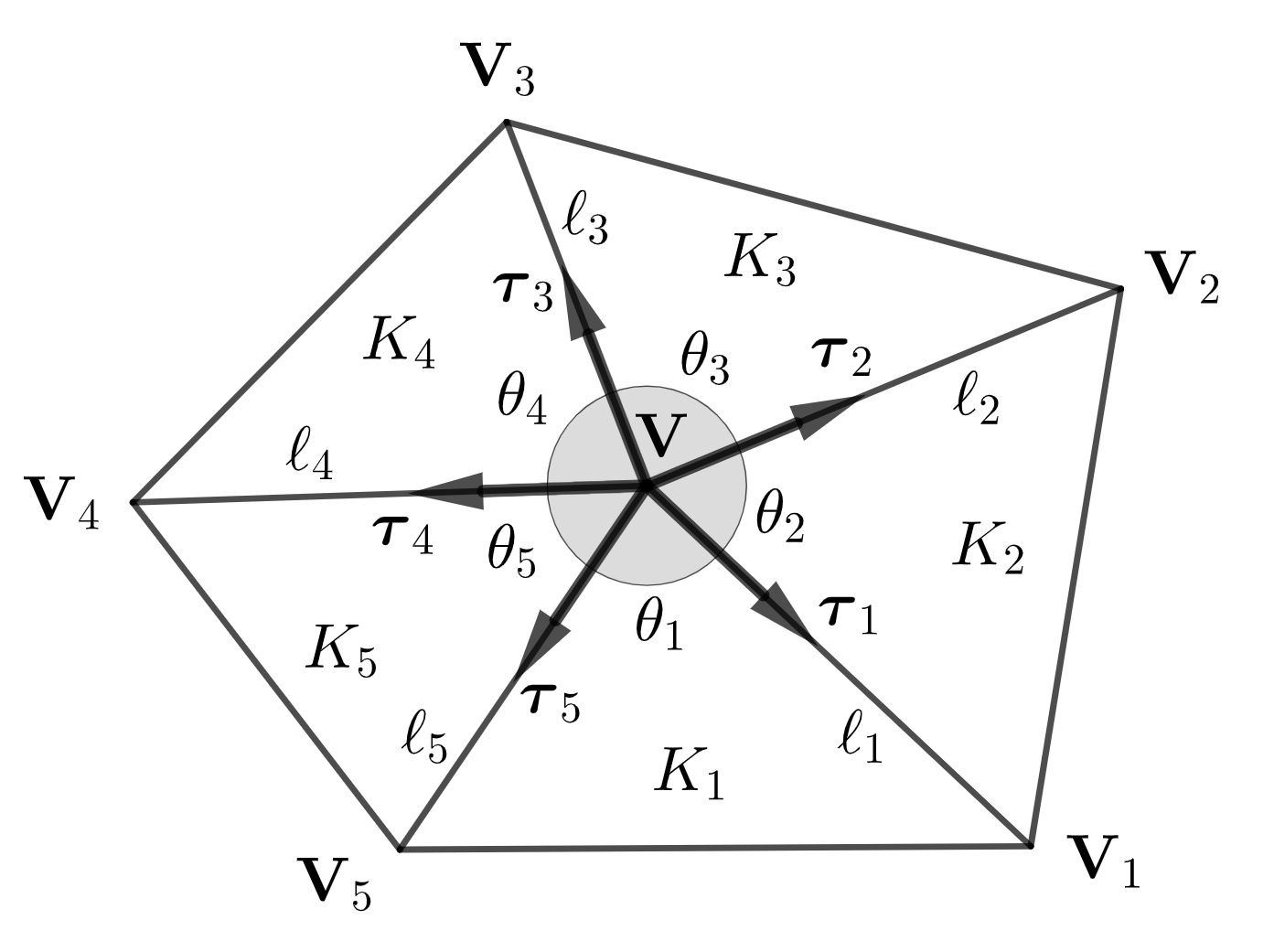} }\qquad
  \subfloat[ boundary regular $\V$, $J=3$, $\JJ=2$ ]{
    \includegraphics[width=0.36\textwidth]{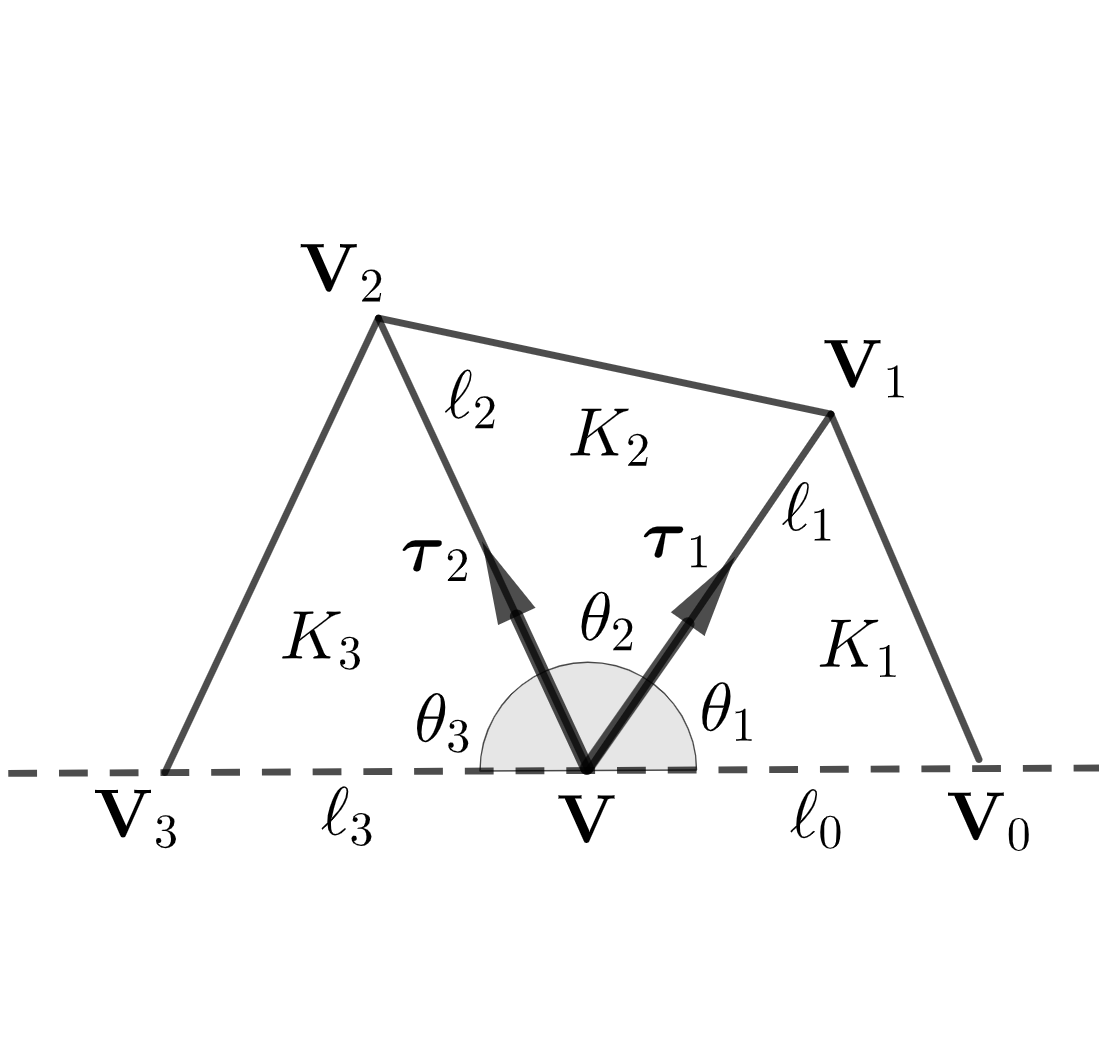}}
  \caption{examples of regular vertices $\V$ (dashed lines belong to $\p\O$)}
  \label{fig:reg}
\end{figure}

\subsection{least square solution of  a system by test functions}\label{sec:not}
For each  $j=1,2,\cdots, \JJ$, similarly to \eqref{cond:w}, there exists a function $w_j\in \mathcal{P}_h^4(\O)\cap H_0^1(\O)$
such that
\begin{equation}\label{cond:wk}
  \frac{\p w_j}{\p\t_j} (\V)=1,\  \frac{\p w_j}{\p\t_j}(\V_j)=0,\
  \int_{\overline{\V\V_j}} w_j\ d\ell=0,\
 \mbox{ the support of  } w_j \mbox{ is } K_j\cup K_{j+1}.
\end{equation}
For the uniqueness of $w_j$, we add the following conditions:
\begin{equation}\label{cond:wk-redun}
  w_j(\G)=0,\ \nabla w_j(\G)={\bf{0}} \mbox{ at each gravity center } \G \mbox{ of }
  K_j, K_{j+1}.
\end{equation}
Then we have $2\JJ$ test functions in $\Xh$ such that
\begin{equation}\label{def:wht}
  \w_h^{\t_j}=w_j\t_j,\quad
  \w_h^{\t_j^\perp}=w_j\t_j^\perp,\quad j=1,2,\cdots,\JJ.
\end{equation}

Consider the following system  of $2\JJ$ equations for unknown $q_h^\V\in \ST(\V)$:
\begin{equation}\label{eq:LSqhv}
  \arraycolsep=1.4pt\def\arraystretch{1.6}
    \begin{array}{ll}
  (q_h^\V, \div\w_h^{\t_j})&=a_j, \\
    (q_h^\V, \div\w_h^{\t_j^\perp})&= b_j,
    \end{array}
  \end{equation}
for given scalars $a_j, b_j, j=1,2,\cdots,\JJ$.
Since  $q_h^{\V}\in \ST(\V)$
 can be represented for $J$ unknown constants $\alp_1, \alp_2,\cdots,\alp_J$ as
 \begin{equation}\label{eq:qhrepre-a}
   q_h^\V=\alp_1\stg{\V}{K_1}+\alp_2\stg{\V}{K_2} +\cdots +
  \alp_J\stg{\V}{K_J},
\end{equation}
the system \eqref{eq:LSqhv} is of $2\JJ$ equations for $J$ unknowns
$\alp_1, \alp_2,\cdots,\alp_J$ in \eqref{eq:qhrepre-a}.

\begin{lemma}\label{lem:singval}
  Let  $q_h^\V\in \ST(\V)$ be the least square solution of the
 system \eqref{eq:LSqhv}.
 Then, if $\V$ is a regular vertex,  we have
\[ \|q_h^\V\|_0 \le   C |\KK(\V)|^{-1/2} \sum_{j=1}^{\JJ}(|a_j| + |b_j|).  \]
\end{lemma}
\begin{proof}
  Let 
 $\theta_j$ be the angle of $K_j$ at $\V$, $j=1,2,\cdots,J$
 and $\ell_j=|\overline{\V\V_j}|, j=0,1,\cdots,J$ as in Figure \ref{fig:reg}.
  Then
  from \eqref{eq:qhdivwh0}, \eqref{cond:wk}, \eqref{def:wht} and \eqref{eq:qhrepre-a},
  we can rewrite \eqref{eq:LSqhv} for $j=1,2,\cdots,\JJ$ as
  \begin{equation}\label{eq:srsys0q}
    \arraycolsep=1.6pt\def\arraystretch{2.2}
   \begin{array}{ll}
   \disp\frac{\ell_{j-1}\ell_j\sin\theta_j}{200}\alp_j &+\disp \frac{\ell_{j}\ell_{j+1}\sin\theta_{j+1}}{200}\alp_{j+1}=a_j,\\ 
      \disp\frac{\ell_{j-1}\ell_j\cos\theta_j}{200}\alp_j &- \disp\frac{\ell_{j}\ell_{j+1}\cos\theta_{j+1}}{200}\alp_{j+1}
                                           =b_j. 
   \end{array}
 \end{equation}
 Denote
 \begin{equation}\label{eq:lem-bet}
   \bet_j=\ell_{j-1}\ell_j\alp_j/200 \quad\mbox{ for } j=1,2,\cdots,J.
 \end{equation}
 Then for $j=1,2,\cdots,\JJ$, we simplify \eqref{eq:srsys0q} into
  \begin{subequations}\label{eq:srsys1q}
 \begin{align}
   \sin\theta_j\bet_j &+ \sin\theta_{j+1}\bet_{j+1}= a_j, \label{eq:srsys1q-a}\\
  \cos\theta_j\bet_j &- \cos\theta_{j+1}\bet_{j+1}=b_j. \label{eq:srsys1q-b}   
   \end{align}
 \end{subequations}
 
 If $\V$ is a regular vertex, then by \eqref{def:nsing}, we can assume without loss of generality,
 \begin{equation}\label{eq:sinth12}
  0< C \le |\sin(\theta_1+\theta_2)|,
 \end{equation}
 which tells the bound of the determinant of two equations \eqref{eq:srsys1q-a},  \eqref{eq:srsys1q-b} for $j=1$.

 Consider the following subsystem of \eqref{eq:srsys1q} consisting of $J$ equations:
 \begin{subequations}\label{eq:srsys2}
 \begin{align}
   \sin\theta_1 x_1 &+ \sin\theta_{2} x_{2}= a_1 ,\label{eq:srsys2-a}\\
  \cos\theta_1 x_1 &- \cos\theta_{2} x_{2}=b_1,\label{eq:srsys2-b}\\   
   \sin\theta_j x_j &+ \sin\theta_{j+1} x_{j+1}= a_j ,\quad j=2,3,\cdots,J-1.
                        \label{eq:srsys2-c}
 \end{align}
 \end{subequations}
 We can solve first \eqref{eq:srsys2-a},\eqref{eq:srsys2-b} to get  $x_1, x_2$.
 Then, solve \eqref{eq:srsys2-c} consecutively to obtain
 \[  x_{j+1}=(\sin\theta_{j+1})^{-1}(a_j-\sin\theta_j x_j),\quad j=2,3,\cdots,J-1.\]
 From \eqref{eq:sinth12} and $ 0<C \le \sin\theta_j, j=1,2,\cdots,J$ by shape-regularity,
 we have
 \begin{equation}\label{eq:singval-bet-bd}
\|(x_1, x_2,\cdots,x_J)\|_2 \le C \|(b_1, a_1,a_2,\cdots,a_{J-1})\|_2. 
 \end{equation}
 
 Let $A\in \R^{J\times J}$ be the matrix for the system \eqref{eq:srsys2}. Then
 from \eqref{eq:singval-bet-bd}, we deduce that
 \begin{equation}\label{eq:singval-inv-A}
   \|A^{-1} \y \|_2 \le C \|\y\|_2\quad \mbox{ for all } \y\in \R^J.
 \end{equation}
 Thus, for the smallest singular value $s$ of $A$, we can estimate the following from \eqref{eq:singval-inv-A}:
 \begin{equation}\label{eq:smallest-s}
   s=\min_{\x\in\R^J\setminus\{\mathbf{0}\}} \frac{\|A\x\|_2}{\|\x\|_2}
   = \min_{\y\in\R^J\setminus\{\mathbf{0}\}}  \frac{\|\y\|_2}{\|A^{-1}\y\|_2} \ge \frac1{C}. 
 \end{equation}
 
 Since
 the smallest singular value of a matrix is greater than that of its submatrix,
 we can estimate
 \begin{equation}\label{eq:lem-qhbdd-0}
   \|(\bet_1,\bet_2,\cdots,\bet_J)\|_2 \le \frac1s
   \|(a_1,a_2,\cdots,a_{\JJ},b_1,b_2,\cdots,b_{\JJ})\|_2,
 \end{equation}
 for the least square solution of \eqref{eq:srsys1q}.
From \eqref{eq:exp-sting} and  \eqref{eq:affine-sting}, we note that
\begin{equation}
  \label{eq:est-sting0}
  \|\st_{\V K_j}\|_0 \le C|K_j|^{1/2}, \quad j=1,2,\cdots, J.
\end{equation}
Thus, combining \eqref{eq:qhrepre-a}, \eqref{eq:lem-bet}, \eqref{eq:smallest-s}-\eqref{eq:est-sting0}, the proof is completed.
\end{proof}

\subsection{definition of $\ph^\V$ for a regular vertex $\V$}
We test \eqref{eq:pihp-0} with  $\w_h^{\t_j}, \w_h^{\t_j^\perp}$ in \eqref{def:wht}.
Then for $\pihp^\V$  in \eqref{eq:decom-pihp-st}, we have 
 \begin{equation}\label{eq:Piphvtk}
    \arraycolsep=1.2pt\def\arraystretch{1.8}
    \begin{array}{l}
   \left(\Pi_hp^\V, \div\w_h^{\t_j}\right)=(\f,\w_h^{\t_j})-(\nabla\u, \nabla\w_h^{\t_j})
   -\left(\Pi_hp^\NS, \div\w_h^{\t_j}\right)-(p-\Pi_hp,\div\w_h^{\t_j}),\\
  \left( \Pi_hp^\V, \div\w_h^{\t_j^\perp}\right)
  =(\f,\w_h^{\t_j^\perp})-(\nabla\u, \nabla\w_h^{\t_j^\perp})
   -\left(\Pi_hp^\NS, \div\w_h^{\t_j^\perp}\right)-(p-\Pi_hp,\div\w_h^{\t_j^\perp}),
    \end{array}
 \end{equation}
 for $j=1,2,\cdots, \JJ$, since $\div\w_h^{\t_j}, \div\w_h^{\t_j^\perp}$ vanish
 at all vertices except $\V$ and
 \[ (1,\div\w_h^{\t_j})_K =  (1,\div\w_h^{\t_j^\perp})_K=0\quad \mbox{ for all } K\in\Th.\]

Reflecting \eqref{eq:Piphvtk}, create the following system of $2\JJ$ equations for unknown $\ph^ \V\in \ST(\V)$:
\begin{subequations}\label{eq:LSphv}
  \begin{align}
  (p_h^\V, \div\w_h^{\t_j})&=(\f,\w_h^{\t_j})-(\nabla\u_h, \nabla\w_h^{\t_j})-(p_h^\NS, \div\w_h^{\t_j}),\label{eq:LSphv-a}\\
    (p_h^\V, \div\w_h^{\t_j^\perp})&=(\f,\w_h^{\t_j^\perp})-(\nabla\u_h, \nabla\w_h^{\t_j^\perp})-(p_h^\NS, \div\w_h^{\t_j^\perp}),\label{eq:LSphv-b}
  \end{align}
\end{subequations}
for $j=1,2,\cdots,\JJ$, where $\ph^\NS$ is the non-sting component, already defined
in Lemma \ref{lem:phns-error}.

\begin{lemma}\label{lem:error-sr0}
  If $\V$ is a regular vertex as in Figure \ref{fig:reg}, define $p_h^\V\in \ST(\V)$
  as the least square solution of \eqref{eq:LSphv}.
  Then, for $\pihp^\V$ in \eqref{eq:decom-pihp-st}, we estimate
  \begin{equation}\label{eq:lem-error-phv}
    \|\Pi_hp^\V-p_h^\V\|_{0} \le C
    \left(|\u-\u_h|_{1,\KK(\V)} 
      +\|p-\Pi_hp\|_{0,\KK(\V)}\right).
  \end{equation}
\end{lemma}
\begin{proof}

If we denote the error by $e_h^\V=\Pi_hp^\V - p_h^\V$,
then from \eqref{eq:Piphvtk} and \eqref{eq:LSphv}, $e_h^\V\in \ST(\V)$ is the least square solution of the following system of $2\JJ$ equations:
\begin{equation}\label{eq:ehvtk}
  \arraycolsep=1.4pt\def\arraystretch{1.6}
  \begin{array}{l}
   (e_h^\V, \div\w_h^{\t_j})=-\left(\nabla(\u-\u_h), \nabla\w_h^{\t_j}\right)
   -\left(\Pi_hp^\NS-p_h^\NS, \div\w_h^{\t_j}\right)-(p-\Pi_hp,\div\w_h^{\t_j}),\\
   (e_h^\V, \div\w_h^{\t_j^\perp})=-\left(\nabla(\u-\u_h), \nabla\w_h^{\t_j^\perp}\right)
   -\left(\Pi_hp^\NS-p_h^\NS, \div\w_h^{\t_j^\perp}\right)
   -(p-\Pi_hp,\div\w_h^{\t_j^\perp}),
  \end{array}
\end{equation} 
for $j=1,2,\cdots, \JJ$, since the least square solution  of a system is the solution
of its normal equation and $\pihp^\V$ can be regarded as the solution of the normal equation of \eqref{eq:Piphvtk}.

By definition of $\w_h^{\t_j}, \w_h^{\t_j^\perp}$ in \eqref{cond:wk}-\eqref{def:wht},
 we note that
 \begin{equation}\label{eq:norm-wh}
   |\w_h^{\t_j}|_1,\ |\w_h^{\t_j^\perp}|_1 \le C(|K_j|+|K_{j+1}|)^{1/2},\quad j=1,2,\cdots,\JJ. 
 \end{equation}
 Thus \eqref{eq:lem-error-phv} comes from \eqref{eq:ehvtk}, \eqref{eq:norm-wh}
and Lemma \ref{lem:phns-error}, \ref{lem:singval}.
\end{proof}

\section{$\ph^\V$ for a nearly singular vertex $\V$, not a dead corner}
\label{sec:sting-sing}
When a vertex $\V$ is exactly singular, the system \eqref{eq:LSphv} is underdetermined, since
the determinant of \eqref{eq:srsys1q-a},  \eqref{eq:srsys1q-b} for each $j=1,2,\cdots,\JJ$,  makes
\[ -\sin(\theta_j +\theta_{j+1}) =-\sin\pi=0.\]
Although $\V$ is not exactly singular,  if it is nearly singular,
the error $\Pi_hp^\V-p_h^\V$  in Lemma \ref{lem:error-sr0}
might be large from the tiny smallest singular value of \eqref{eq:ehvtk}.

To overcome the problem on nearly singular vertices,
we will replace the equations in \eqref{eq:LSphv-b}
with new ones
utilizing jumps of $\ph^\NS$  and $\ph^{\V}$
for regular vertices $\V$
defined in Lemma \ref{lem:phns-error} and \ref{lem:error-sr0}, respectively.

\subsection{jump of tangential derivative}
For simple motivation, let's fix a boundary vertex $\V$
which is not a corner point of $\p\O$. If $\V$ is nearly singular,
it is exact singular and has only two triangles $K_1, K_2$ which share $\V$ as in Figure \ref{fig:sing}-(a).
\begin{figure}[h]
  \centering
    \subfloat[boundary singular non-corner $\V$]{
    \includegraphics[width=0.3\textwidth]{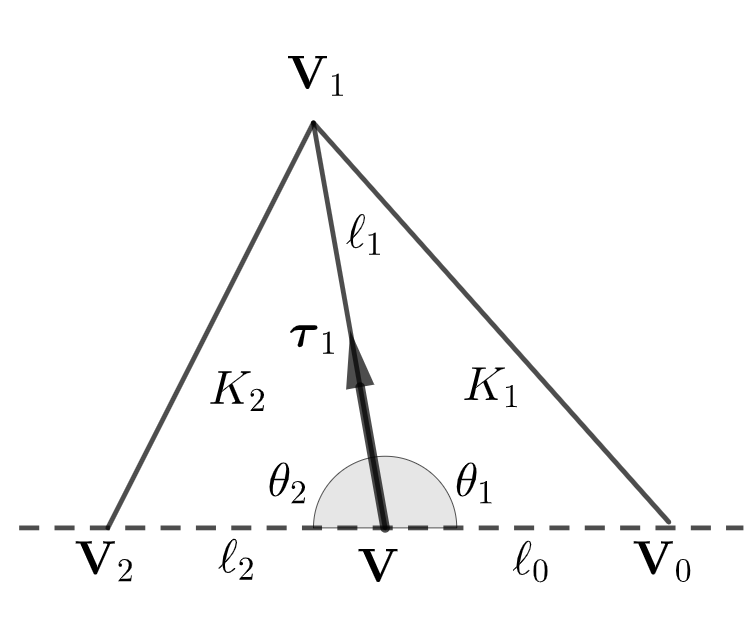} }
   \subfloat[nearly singular corner $\V$, not dead]{
     \includegraphics[width=0.33\textwidth]{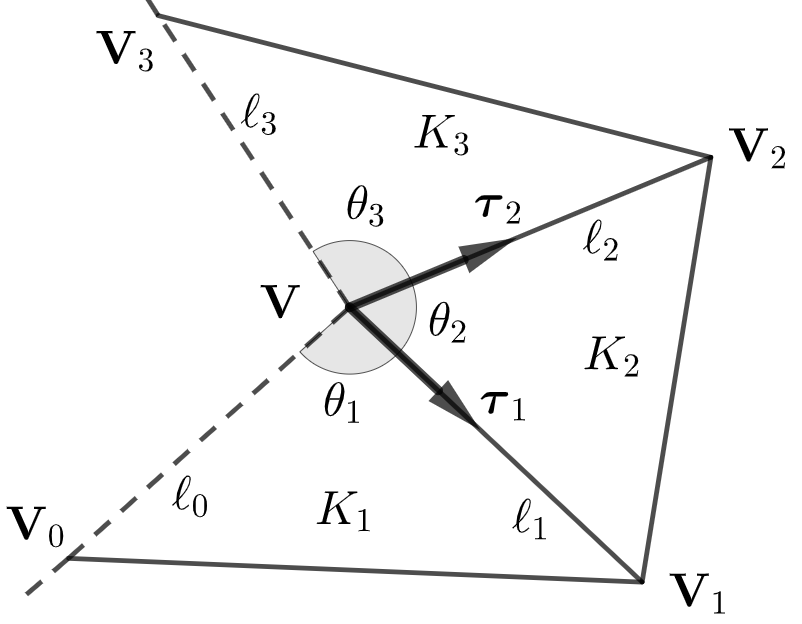}}
    \subfloat[interior nearly singular $\V$]{
    \includegraphics[width=0.33\textwidth]{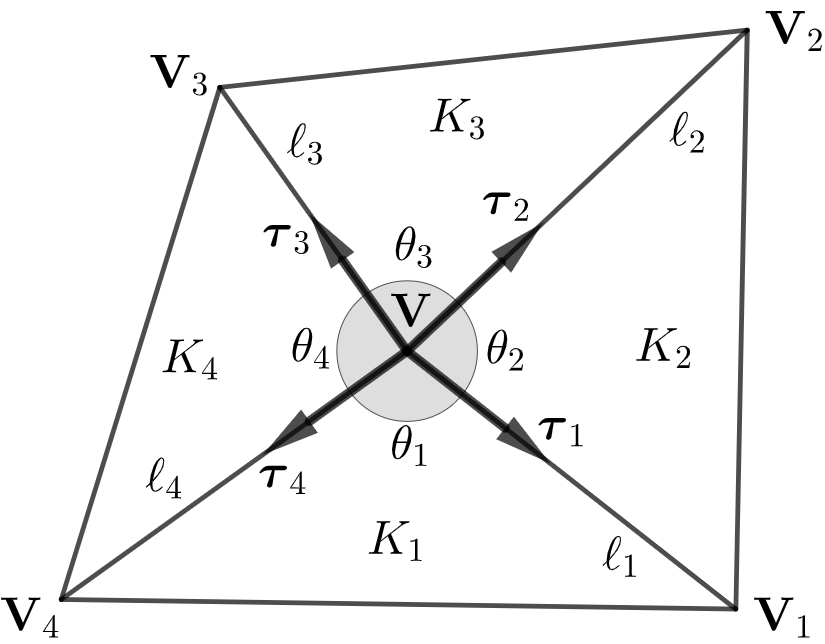}}
  \caption{examples of nearly singular ordinary vertices $\V$ (dashed lines belong to $\p\O$) }
  \label{fig:sing}
\end{figure}

Then,
since the system \eqref{eq:LSphv} is singular, in order to define $\ph^\V\in\ST(\V)$ approximating $\pihp^\V$, we have to create a new equation for $\ph^\V$, reflecting some condition
for  $\pihp^\V$.

Define a jump of a function $q_h$   at $\V$ across $K_1\cap K_2$ as
\begin{equation}\label{def:jump12}
  \jump_{12}(q_h)
  ={|K_1\cap K_2|^3}\left(\frac{\p}{\p \t_1} \left(q_h\big|_{K_1}\right) (\V) - \frac{\p}{\p \t_1}\left(q_h\big|_{K_2}\right) (\V)\right).
\end{equation}
Then, since $\pihp$ is continuous on $K_1\cap K_2$, we have $\jump_{12}(\pihp)=0$. It is written in
\begin{equation}\label{eq:jump-php}
\jump_{12}(\pihp)=  \jump_{12}\left( \pihp^\NS+  \pihp^{\V} +  \pihp^{\V_0} +  \pihp^{\V_1} +  \pihp^{\V_2}
  \right)=0.
\end{equation}

We note that $\stg{\V_0}{K_1}$, $\stg{\V_2}{K_2}$ are constant
on $K_1\cap K_2$ from \eqref{eq:affine-sting}, \eqref{eq:value-sting}.
 It results in
\begin{equation}\label{eq:jump-php-svb}
  \jump_{12}\left(  \pihp^{\V_0} \right) = \jump_{12}\left(  \pihp^{\V_2} \right)=0.
\end{equation}
Thus, from \eqref{eq:jump-php} and \eqref{eq:jump-php-svb}, $\pihp^\V$ satisfies
\begin{equation}\label{eq:bdy-sing-0}
  \jump_{12}\left(\pihp^\V\right)= -\jump_{12}\left(\pihp^\NS+\pihp^{\V_1}\right).
\end{equation}

We note $\V_1$ is a regular vertex by Lemma \ref{lem:isol-vtx}, 
since
$\V_1$ and $\V$ are connected by an interior edge.
It means that we have $\ph^{\V_1}$, already defined in Lemma \ref{lem:error-sr0}.

Thus we can impose a following new condition for unknown $\ph^\V$,
which is similar to \eqref{eq:bdy-sing-0}:

\begin{equation}\label{eq:jump-ph-simil}
  \jump_{12}(\ph^\V) = -\jump_{12}(\ph^\NS+p_h^{\V_1}).
\end{equation}
Then, replacing \eqref{eq:LSphv-b} with \eqref{eq:jump-ph-simil},
consider the following system for unknown $\ph^\V\in\ST(\V)$:
  \begin{equation}\label{eq:SSbdphv} \arraycolsep=1.4pt\def\arraystretch{1.6}
  \begin{array}{rl}
  (p_h^\V, \div\w_h^{\t_1})&=(\f,\w_h^{\t_1})-(\nabla\u_h, \nabla\w_h^{\t_1})-(p_h^\NS, \div\w_h^{\t_1}),\\
   \jump_{12} (p_h^\V) &= -\jump_{12}(\ph^{\nsg{}{}}+p_h^{\V_1}).
  \end{array}
\end{equation}
\begin{lemma}\label{lem:phv-error-svb}
  If $\V$ is a nearly singular vertex on  $\p\O$, not a corner as in Figure \ref{fig:sing}-(a),
  define $p_h^\V\in\ST(\V)$ as the solution of \eqref{eq:SSbdphv}.
  Then,  for $\pihp^\V$ in \eqref{eq:decom-pihp-st}, we estimate
\begin{equation}\label{eq:lem-error-phvs-svb}
    \|\Pi_hp^\V-p_h^\V\|_{0} \le C
    \left(|\u-\u_h|_{1,\KK(\V_1)} +\|p-\Pi_hp\|_{0,\KK(\V_1)}\right),
  \end{equation}
  where $\V_1$ is a regular vertex which shares an interior edge with $\V$.
\end{lemma}
\begin{proof}
Denote $e_h^\V=\Pi_hp^\V - p_h^\V$.
Then from \eqref{eq:bdy-sing-0}, \eqref{eq:SSbdphv}
and the same argument inducing \eqref{eq:ehvtk}, we have
\begin{equation}\label{eq:ehvtk-sbd}
  \arraycolsep=1.4pt\def\arraystretch{1.6}
  \begin{array}{rll}
   (e_h^\V, \div\w_h^{\t_1})&=&-\left(\nabla(\u-\u_h), \nabla\w_h^{\t_1}\right)
   -\left(\Pi_hp^\NS-p_h^\NS, \div\w_h^{\t_1}\right)-(p-\Pi_hp,\div\w_h^{\t_1}),\\
    \jump_{12}(e_h^\V)&=&-\jump_{12}\left(\pihp^\NS- \ph^\NS\right)
  -\jump_{12}\left( \pihp^{\V_1}-\ph^{\V_1}\right).
  \end{array}
\end{equation} 

We can  represent $e_h^\V\in\ST(\V)$ with $2$ unknown constants $e_1, e_2$ as 
\begin{equation}\label{def:ehv-bs}
e_h^\V=e_1 \stg{\V}{K_1} + e_2\stg{\V}{K_2}.
\end{equation}
Let 
 $\theta_j$ be the angle of $K_j$ at $\V$, $j=1,2$
 and $\ell_j=|\overline{\V\V_j}|, j=0,1,2$ as in Figure \ref{fig:sing}-(a).
Then, from \eqref{eq:exp-sting}, \eqref{eq:affine-sting},
we calculate
\begin{equation}\label{eq:J12-repre}
    \jump_{12} (e_h^\V)= -6{\ell_1}^2e_1 +6{\ell_1}^2e_2.
  \end{equation}
Abbreviating 
the right hand sides in \eqref{eq:ehvtk-sbd} by $a$, $b$, respectively,
 we can rewrite \eqref{eq:ehvtk-sbd} by \eqref{eq:qhdivwh0}, \eqref{eq:J12-repre} into
\begin{equation}\label{eq:ehvtk-sbd-2}\arraycolsep=1.4pt\def\arraystretch{1.6}
 \begin{array}{ccccl}
   \disp\frac{\ell_0\ell_1\sin\theta_1}{200}e_1 &+&\disp
                                                    \frac{\ell_{1}\ell_{2}\sin\theta_{2}}{200}e_{2} &=&a, \\
   -6{\ell_1}^2e_1 &+&6{\ell_1}^2e_2 &=&b.
 \end{array}
\end{equation}
From Lemma \ref{lem:phns-error} and \eqref{eq:norm-wh},
we have
\begin{equation}\label{est:norm-a} |a|\le  C \left(|K_1|+|K_2|\right)^{1/2}
  \left(|\u-\u_h|_{1,K_1\cup K_2} +\|p-\Pi_hp\|_{0,K_1\cup K_2}\right).
\end{equation}

For $q_h\in\Mh$ in \eqref{def:jump12}, we can estimate
\begin{equation}\label{est:jumpqh}
    \arraycolsep=1.5pt\def\arraystretch{2.0}
  \begin{array}{lll}
   \left| \jump_{12}\left( q_h\right)\right| &\le&
{\ell_1}^3  \left( \left\|\nabla q_h \big|_{K_1}(\V)\right\|_2
           +  \left\|\nabla q_h \big|_{K_2}(\V)\right\|_2 \right)  
                                                           \le C{\ell_1}^3 \left(|K_1|+|K_2|\right)^{-1/2}\left|q_h\right|_{1,K_1\cup K_2}\\
                                             &\le& C{\ell_1}^2 \left|q_h\right|_{1,K_1\cup K_2}\le C{\ell_1} \left\|q_h\right\|_{0,K_1\cup K_2}  \le C \left(|K_1|+|K_2|\right)^{1/2}\left\|q_h\right\|_{0,K_1\cup K_2}.
  \end{array}                                                        
\end{equation}
Then, since $\V_1$ is regular, Lemma \ref{lem:phns-error}, \ref{lem:error-sr0}
and \eqref{est:jumpqh} deduce
\begin{equation}\label{est:norm-jump}
 |b|\le  C \left(|K_1|+|K_2|\right)^{1/2}
  \left(|\u-\u_h|_{1,\KK(\V_1)} +\|p-\Pi_hp\|_{0,\KK(\V_1)}\right).
\end{equation}

We note that the system \eqref{eq:ehvtk-sbd-2} is far from singular,
since $0<C \le \sin\theta_j, j=1,2$ by shape-regularity.
Thus from \eqref{eq:ehvtk-sbd-2}, \eqref{est:norm-a}, \eqref{est:norm-jump}, we have
\begin{equation}\label{eq:phv-error-phvs-svb-9}
  |e_1|+|e_2|\le C \left(|K_1|+|K_2|\right)^{-1/2}
  \left(|\u-\u_h|_{1,\KK(\V_1)} +\|p-\Pi_hp\|_{0,\KK(\V_1)}\right).
\end{equation}
Then  \eqref{eq:lem-error-phvs-svb} comes
from \eqref{eq:est-sting0}, \eqref{def:ehv-bs}, \eqref{eq:phv-error-phvs-svb-9}.
\end{proof}

\subsection{definition of $\ph^\V$ for a nearly singular ordinary vertex $\V$}
Let $\V$ be a nearly singular vertex meeting an interior edge
and $K_1, K_2, \cdots, K_J$, all triangles in $\Th$ sharing $\V$ as in Figure \ref{fig:sing}.
Using the same notations in  \eqref{def:JJ}-\eqref{def:v0vJ},
 define a jump of a function $q_h$ at $\V$ across an interior edge $K_j\cap K_{j+1}$  as
\begin{equation}\label{def:jumpkk1}
  \jump_{\jjp}(q_h)
  ={|K_j\cap K_{j+1}|^3}\left(\frac{\p}{\p \t_j} \left(q_h\big|_{K_j}\right) (\V) - \frac{\p}{\p \t_j}\left(q_h\big|_{K_{j+1}}\right) (\V)\right),
\end{equation}
for  $j=1,2,\cdots,\JJ$, similarly to \eqref{def:jump12}.

We note that the adjacent vertices $\V_1, \V_2,\cdots,\V_{\JJ}$
are all regular from Lemma \ref{lem:isol-vtx}.
That is, $\ph^{\V_1}, \ph^{\V_2},\cdots, \ph^{\V_\JJ}$ are already defined
in Lemma \ref{lem:error-sr0}.

Thus, 
replacing \eqref{eq:LSphv-b} with new equations using the jumps 
in \eqref{def:jumpkk1}, 
we can consider the following system of $2\JJ$ equations for unknown $\ph^\V\in\ST(\V)$:
  \begin{equation}\label{eq:SSphv-svi} \arraycolsep=1.4pt\def\arraystretch{1.6}
  \begin{array}{lll}
  (p_h^\V, \div\w_h^{\t_j})&=&(\f,\w_h^{\t_j})-(\nabla\u_h, \nabla\w_h^{\t_j})-(p_h^\NS, \div\w_h^{\t_j}),\\
    \jump_{\jjp} (p_h^\V) &=& -\jump_{\jjp}(\ph^\NS+p_h^{\V_j} ),\quad
  \end{array}
\end{equation}
for  $j=1,2,\cdots,\JJ$.

Then, we can repeat the arguments for Lemma \ref{lem:singval}, \ref{lem:error-sr0},
\ref{lem:phv-error-svb} to establish the following lemma.
\begin{lemma}\label{lem:phv-error-svi}
  If $\V$ is a nearly singular vertex meeting an interior edge as in Figure \ref{fig:sing},
  define $\ph^\V\in\ST(\V)$ as the least square solution of
  \eqref{eq:SSphv-svi}. Then, for $\pihp^\V$ in \eqref{eq:decom-pihp-st},  we estimate
\begin{equation*}
    \|\Pi_hp^\V-p_h^\V\|_{0} \le C \sum_{j=1}^{\JJ}
    \left(|\u-\u_h|_{1,\KK(\V_j)} +\|p-\Pi_hp\|_{0,\KK(\V_j)}\right),
  \end{equation*}
where $\V_1, \V_2, \cdots, \V_{\JJ}$ are all vertices sharing interior edges with $\V$.
\end{lemma}

\section{$\ph^\V$ for a dead corner $\V$}\label{sec:dead}
\label{sec:sting-corner}
Let $\V$ be a vertex meeting no interior edge. 
Then $\V$ is a dead corner and has only one triangle $K_1$  as in Figure \ref{fig:dead}.
There exists a triangle $K$ in $\Th$ sharing 
two vertices $\W_1, \W_2$ with $K_1$ . Denote by $\W_3$,
the third vertex of $K$ not shared with $K_1$.

\begin{figure}[h]
  \centering
  \includegraphics[width=0.32\textwidth]{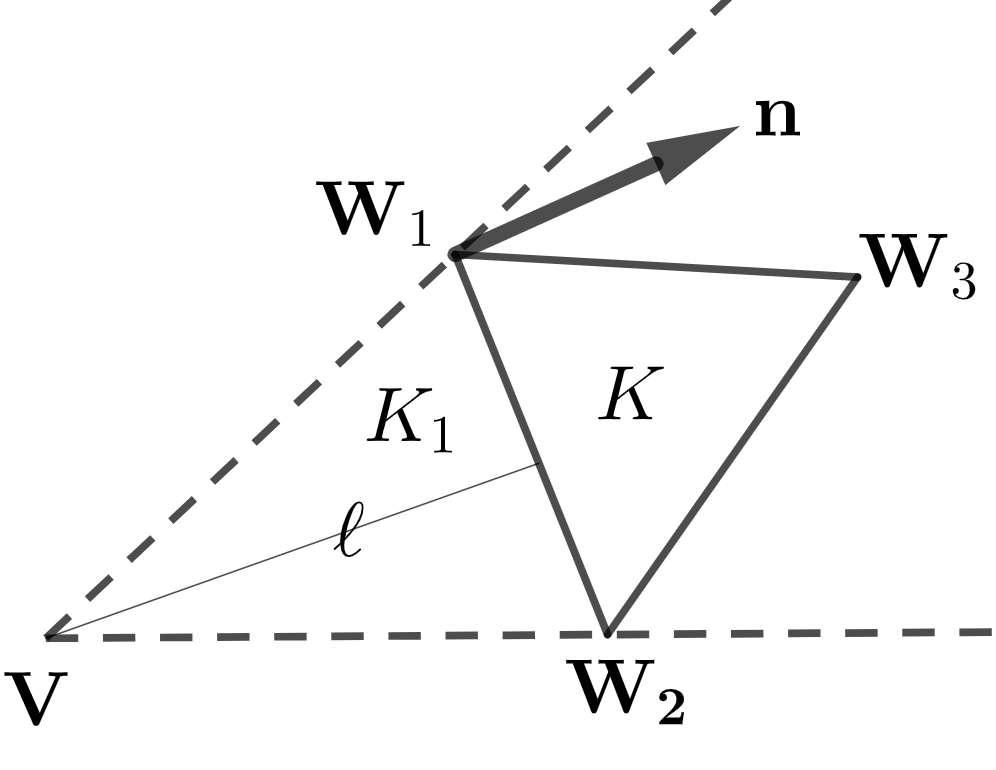}
  \caption{an example of a dead corner $\V$ (dashed lines belong to $\p\O$)}
  \label{fig:dead}
\end{figure}

Define a jump of a function $q_h$  at $\W_1$ across $K_1\cap K$ as
\begin{equation}\label{def:jump12n}
  \jump(q_h)
  ={\ell^3}\left(\frac{\p}{\p \n} \left(q_h\big|_{K_1}\right) (\W_1)
    - \frac{\p}{\p \n}\left(q_h\big|_{K}\right) (\W_1)\right),
\end{equation}
where $\n$ is a unit outward normal vector on $K_1\cap K$ of $K_1$ and $\ell$ is the distance
between $\V$ and $K_1\cap K$.

We note that  $\ph^{\W_1}, \ph^{\W_2} $, $\ph^{\W_3}$ are already defined 
in Lemma \ref{lem:error-sr0} and \ref{lem:phv-error-svi}, since $\W_1, \W_2, \W_3$ are not corners  by Assumption \ref{asm:Th} on $\Th$.
Thus we can consider the following equation for  unknown $\ph^\V\in \ST(\V)$,
\begin{equation}\label{eq:lastphv}
  \jump(\ph^\V)=-\jump(\ph^\NS+\ph^{\W_1}+\ph^{\W_2}+\ph^{\W_3}).
\end{equation}

For a vertex $\W$, let $\mathcal{V}(\W)$ be a set of all vertices
sharing interior edges with $\W$, then denote
\[ \KK\KK(\W) = \bigcup_{\mathbf{U}\in\mathcal{V}(\W)} \KK(\mathbf{U}).  \]
\begin{lemma}\label{lem:phv-error-corner}
  If $\V$ is a vertex meeting no interior edge as in Figure \ref{fig:dead},
  define $p_h^\V\in\ST(\V)$ as the solution of \eqref{eq:lastphv}.
  Then, for $\pihp^\V$ in \eqref{eq:decom-pihp-st}, we estimate
\begin{equation}\label{eq:lem:phv-error-corner-0}
  \|\Pi_hp^\V-p_h^\V\|_{0} \le C \sum_{j=1}^3
    \left(|\u-\u_h|_{1,\KK\KK(\W_j)} +\|p-\Pi_hp\|_{0,\KK\KK(\W_j)}\right),
  \end{equation}
  where $\W_1, \W_2, \W_3$ are all vertices of the triangle whose intersection 
  with $\KK(\V)$ is an edge.
\end{lemma}
\begin{proof}
We remind that $\nabla\pihp$ is continuous at $\W_1$ by \eqref{def:php}.
It can be written in
\begin{equation}\label{eq:jump-php-w}
  \jump\left( \pihp^\V   \right) =-\jump\left( \pihp^\NS +\pihp^{\W_1}+\pihp^{\W_2}+\pihp^{\W_3}\right).
\end{equation}
Set
the error $e_h^\V=\pihp^\V-\ph^\V=e\stg{\V}{K_1}$ for some constant $e$,
then \eqref{eq:lastphv} and \eqref{eq:jump-php-w} make
\begin{equation}
  \label{eq:jump-last-phpph}
  \jump(e\stg{\V}{K_1}) = -\jump\left(   \pihp^{\nsg{}{}}-\ph^{\nsg{}{}} +
\sum_{j=1}^3\pihp^{\W_j}-\ph^{\W_j}   \right).
\end{equation}

From \eqref{eq:exp-sting}, \eqref{eq:affine-sting}, we have
\begin{equation}
  \label{eq:lem:phv-corner-a}
   \jump(\stg{\V}{K_1})=-\frac95{\ell}^2.
 \end{equation}
For the right hand side in \eqref{eq:jump-last-phpph},
we can estimate the following for  $q_h\in\Mh$,
\begin{equation}\label{est:jumpqh2}
  \left| \jump\left( q_h\right)\right| \le C \left(|K_1|
    +|K|\right)^{1/2}\left\|q_h\right\|_{0,K_1\cup K},
\end{equation}
 same as in \eqref{est:jumpqh}
by similarity of \eqref{def:jump12} and \eqref{def:jump12n}.

Then, we can deduce \eqref{eq:lem:phv-error-corner-0}
from \eqref{eq:est-sting0}, \eqref{eq:jump-last-phpph}-\eqref{est:jumpqh2} and
Lemma \ref{lem:phns-error}, \ref{lem:error-sr0}, \ref{lem:phv-error-svi}.
\end{proof}

We have defined the sting components $\ph^\V\in \ST(\V)$ for all vertices $\V$
in Lemma \ref{lem:error-sr0}, \ref{lem:phv-error-svi} and \ref{lem:phv-error-corner}.
All the results are summarized in the following lemma.
\begin{lemma}\label{lem:phst-error}
  Define a sting component $\ph^\ST\in \ST_h$ as
  \begin{equation*}
    \ph^{\ST} = \sum_{\V:\mbox{\rm{vertex}}} \ph^{\V}.
\end{equation*}
  Then, for the sting component $\pihp^\ST$ of $\pihp$ in \eqref{eq:decom-pihp},
  we estimate
  \begin{equation*}
    \|\pihp^\ST-\ph^\ST\|_0 \le C  \left(|\u-\u_h|_{1} +\|p-\Pi_hp\|_{0}\right).
\end{equation*}
\end{lemma}

\section{Piecewise constant component}\label{sec:const}
\subsection{inf-sup condition}
Define the following spaces:
\begin{align}
  \label{def:vtx-c0}
  \begin{split}
    V_{h,0}&=\{\v_h\in [\PO4\cap H_0^1(\O)]^2\ :\ \nabla\v_h \mbox{ is   continuous at all vertices in }\Th\},\\
    V_{h,00}&=\{\v_h\in V_{h,0}\ :\ \nabla\v_h \mbox{ vanishes at all vertices in }\Th\}.
\end{split}
\end{align}
Then we have the following inf-sup condition for $V_{h,00}\times\PO0\cap L_0^2(\O)$.
\begin{lemma}\label{lem:V00-infsup}
  For each $\cs_h\in\PO0\cap L_0^2(\O)$,  there exists a nontrivial
  $\v_h\in V_{h,00}$ such that
 \begin{equation*}
    \beta \|\cs_h\|_0 |\v_h|_1 \le (\cs_h,\div\v_h),
  \end{equation*}
  for a constant $\bet>0$ regardless of $h$.
\end{lemma}
\begin{proof}
  Given $\cs_h\in\PO0\cap L_0^2(\O)$, there exists a nontrivial $\w\in [\PO2 \cap H_0^1(\O)]^2$
  such that \cite{Bernardi1985}
  \begin{equation}
    \label{eq:p2-p0-infsup}
   \beta \|\cs_h\|_0 |\w|_1 \le (\cs_h,\div\w), 
 \end{equation}
 for a constant $\bet>0$ regardless of $h$.
 
For each triangle $K$ in $\Th$, define $\z_K\in [P^4]^2$ so that
 \begin{equation}
   \label{def:v-in-p4}\arraycolsep=1.4pt\def\arraystretch{1.6}
   \begin{array}{rcll}
     \nabla \z_K &=&\nabla \w &\mbox{ at all vertices of }K,\\
    \disp \int_E \z_K\ d\ell &=& \bf0 &\mbox{ for each edge } E \mbox{ of }K,\\
     \z_K &=&\bf0 &\mbox{ at all 3 vertices and midpoints of 3 medians of } K.
   \end{array}
 \end{equation}
 Then, for a reference triangle $\hat K$ and an affine map $F:\hat K\longrightarrow K$, we have
 \begin{equation}
   \label{eq:v-norm-w}
   |\z_K|_{1,K} \le C |\z_K\circ F|_{1,\hat K} \le  C |\w\circ F|_{1,\hat K} \le C |\w|_{1,K}.
 \end{equation}

 If we define $\z\in [\PO4]^2$ by $\z\big|_K=\z_K$ for all $K\in\Th$,
 then $\z$ belongs to $[H_0^1(\O)]^2$,
 since derivatives of $\w$ along to edges are continuous.
  We note that $\w\neq\z$. If so, $\w={\bf{0}}$
  from the second and third conditions in \eqref{def:v-in-p4}.
  
 Thus, from \eqref{def:v-in-p4} and \eqref{eq:v-norm-w}, $\z$ satisfies
 \begin{equation}
   \label{eq:prop-v-vec}
   \w-\z\in V_{h,00}\setminus\{{\bf{0}}\},\quad |\z|_1 \le C |\w|_1,\quad
    (1,\div\z)_K=0 \mbox{ for all } K.
  \end{equation} 
  Then, the following comes from \eqref{eq:p2-p0-infsup} and \eqref{eq:prop-v-vec}, which completes the proof:
 \begin{equation*}
    \beta/(1+C) \|\cs_h\|_0 |\w-\z|_1\le \beta \|\cs_h\|_0 |\w|_1 \le (\cs_h,\div\w)
   = \left(\cs_h,\div(\w-\z)\right). 
 \end{equation*}
\end{proof}

\subsection{definition of piecewise constant component}
\begin{lemma}\label{lem:exist-phc}
  There exists a unique $\ph^\CC\in\PO0\cap L_0^2(\O)$ satisfying
  \begin{equation}
    \label{def:ph-const}
    (\ph^\CC,\div\v_h)=(\f,\v_h)-(\nabla\u_h,\nabla\v_h)-(p_h^\NS,\div\v_h)\quad
    \mbox{ for all } \v_h\in V_{h,00}.
  \end{equation}     
\end{lemma}   
\begin{proof}
 The uniqueness comes from the inf-sup condition in Lemma  \ref{lem:V00-infsup}.
 For the existence, let
  \begin{align}
    \label{def:vtx-c0q}
    \begin{split}
      \QQ_{h,0}&=\{q_h\in \PO3\cap L_0^2(\O)\ :\ q_h \mbox{ is continuous at all vertices in }\Th\},\\
      {\QQ_{h,00}}&=\{q_h\in \QQ_{h,0}\ :\ q_h \mbox{ vanishes at all corners of }\p\O\}.
    \end{split}
  \end{align}
  Then, there exists a unique $(\widetilde{\u}_h,r_h)\in V_{h,0}\times  {\QQ_{h,00}}$ satisfying
  the following discrete Stokes equation:
  \begin{equation}
  \label{eq:uhrh-condi}
  (\nabla\widetilde{\u}_h, \nabla\v_h) +(r_h,\div\v_h) +(q_h,\div\widetilde{\u}_h) =(\f,\v_h)\quad\mbox{ for all }
  (\v_h,q_h)\in V_{h,0}\times  {\QQ_{h,00}},
\end{equation}
since the following Stokes complex is exact for
$\Sigma_{h,0}, V_{h,0}, {\QQ_{h,00}}$ in \eqref{def:Argyris}, \eqref{def:vtx-c0}, \eqref{def:vtx-c0q}, respectively \cite{Falk2013}:
\begin{equation*}
  0 \longrightarrow \Sigma_{h,0}\xrightarrow{\curl} V_{h,0}
  \xrightarrow{\div} {\QQ_{h,00}} \longrightarrow 0.
\end{equation*}

From \eqref{def:zh0}, we note that $\widetilde{\u}_h\in V_{h,0}$ in  \eqref{eq:uhrh-condi} coincides with
the previous $\u_h\in \Z_{h,0}$ in \eqref{eq:dclv-SC}.
Thus $r_h\in{\QQ_{h,00}}$ in  \eqref{eq:uhrh-condi} satisfies that
\begin{equation}
  \label{cond:rh-main}
  (r_h,\div\v_h)=(\f,\v_h)-(\nabla\u_h,\nabla\v_h)\quad\mbox{ for all } \v_h\in V_{h,0}. 
\end{equation}

Similarly to \eqref{eq:decom-pihp}, decompose $r_h$ into
\begin{equation}
  \label{eq:decom-rh}
  r_h=  r_h^\NS  + r_h^\ST + r_h^\CS\quad\mbox{ for }
  r_h^{\NS}\in \NS_h,\ r_h^\ST \in \ST_h, \ r_h^{\CS}\in \CS_h.
\end{equation}
From the quadrature rule in \eqref{def:sting} and  definition of $V_{h,00}$ in \eqref{def:vtx-c0}, we have
   \begin{equation}
     \label{eq:rhs-v00}
     (r_h^\ST, \div\v_h)=0\quad\mbox{ for all } \v_h\in V_{h,00}.
   \end{equation}   
Then, by  \eqref{def:STNSCS} and \eqref{cond:rh-main}-\eqref{eq:rhs-v00},
     $r_h^\CS-\m(r_h^\CS)\in\PO0\cap L_0^2(\O)$ satisfies
     \begin{equation}\label{eq:rhc-vh00}
       \left(r_h^\CS-\m(r_h^\CS),\div\v_h\right)
       =(\f,\v_h)-(\nabla\u_h,\nabla\v_h)-(r_h^\NS,\div\v_h)\quad
       \mbox{ for all } \v_h\in V_{h,00}.
     \end{equation}
     
For each triangle $K$, we note $\BB(K)\subset V_{h,00}$. 
Thus, from \eqref{def:nonsting}, \eqref{eq:prop-tb}, \eqref{eq:rhc-vh00}, we have
\begin{equation}
  \label{cond:rhK}
   (r_h^\NS\big|_K,\div\v_h)=(\f,\v_h)-(\nabla\u_h,\nabla\v_h)\quad\mbox{ for all }
    \v_h\in \BB(K).
  \end{equation}
Since $\ph^K\in\NS(K)$ satisfying \eqref{cond:phK} is unique,
\eqref{cond:rhK} implies that $r_h^\NS$ in \eqref{eq:rhc-vh00} coincides with $\ph^\NS$
in \eqref{def:ph-const}.
   \end{proof}  
   \begin{lemma}\label{lem:phcc-error}
     Define the piecewise constant component as  $\ph^\CC$ in Lemma \ref{lem:exist-phc}.
     Then for $\pihp^\CS$ in \eqref{eq:decom-pihp}, we estimate
     \begin{equation}\label{eq:lem-phcc-error}
       \|\pihp^\CS-\m(\pihp^\CS) -\ph^{\CC}\|_0  \le  C (|\u-\u_h|_1 + \|p-\pihp\|_0).
     \end{equation}
   \end{lemma}
   \begin{proof}
     By  \eqref{eq:pihp-0} and quadrature rule in  \eqref{def:sting},
     $\pihp^\CS$ satisfies for all $\v_h\in V_{h,00}$,
     \begin{equation}
       \label{eq:lem-phcc}
       (\pihp^\CS,\div\v_h)=(\f,\v_h)-(\nabla\u,\nabla\v_h)-(\pihp^\NS,\div\v_h)-(p-\pihp,\div\v_h).
     \end{equation}
    Set $e_h=\pihp^\CS-\m(\pihp^\CS) -\ph^{\CC}\in \PO0\cap L_0^2(\O)$.
    Then, from  \eqref{def:ph-const} and \eqref{eq:lem-phcc}, it satisfies
     \begin{equation}
       \label{eq:lem-phcc-2}
       (e_h,\div\v_h)=-(\nabla\u-\nabla\u_h,\nabla\v_h)-(\pihp^\NS-\ph^\NS,\div\v_h)-(p-\pihp,\div\v_h),
     \end{equation}
     for all  $\v_h\in V_{h,00}$.
     
  From Lemma  \ref{lem:V00-infsup}, there exists a nontrivial $\v_h\in V_{h,00}$ such that
  \begin{equation}
    \label{eq:lem-phcc-22}
    \beta \|e_h\|_0 |\v_h|_1 \le (e_h,\div\v_h)\quad \mbox{ for }\beta>0 \mbox{ regardless of } h.
  \end{equation}
  Then \eqref{eq:lem-phcc-error} comes from
  \eqref{eq:lem-phcc-2}, \eqref{eq:lem-phcc-22} and Lemma \ref{lem:phns-error}.
\end{proof}

\subsection{local calculation of piecewise constant component }
For each triangle $K\in\Th$, let $C_K$ be a constant such that
\begin{equation} \label{eq:qhc-c1c2}
  C_K=\ph^\CC\big|_{K}.
\end{equation}
If $K_1, K_2$ are two adjacent triangles sharing an edge,
there exists a test function $\v_h\in V_{h,00}$ such that
\begin{equation}
  \label{cond:vh-for-qhc}
\mbox{ the support of } \v_h \mbox{ is in } K_1\cup K_2,\quad   \int_{K_j}\div\v_h\ dxdy=(-1)^{j+1}, j=1,2.
 \end{equation}
 Then from    \eqref{def:ph-const}, \eqref{eq:qhc-c1c2}, \eqref{cond:vh-for-qhc},
 we can calculate the adjacent difference by
\begin{equation}
  \label{eq:lhs-qhc-c1c2}
  C_{K_1}-C_{K_2}=(\f,\v_h)-(\nabla\u_h,\nabla\v_h)-(p_h^\NS,\div\v_h).
\end{equation}

Fix a triangle $K_0$ and denote $C_0=\ph^\CC\big|_{K_0}$.
Then we can calculate $C_K-C_0$ for all triangles $K\in\Th$
by the following iterative algorithm.
\begin{enumerate}
\item[(i)] Set $\widetilde\O=K_0$ and $\mathcal{V}=\emptyset$.
\item[(ii)] Choose a vertex $\V\notin \mathcal V$ on the boundary of $\widetilde\O$.
\item[(iii)]  Calculate $C_K-C_0$ for all triangles $K\subset\KK(\V)$
  by adding adjacent differences in \eqref{eq:lhs-qhc-c1c2}.
\item[(iv)] Update $\widetilde\O=\widetilde\O\cup \KK(\V)$, $\mathcal{V}=\mathcal{V}\cup\{\V\}$
  and go to (ii), if $\widetilde\O\neq\O$.
\end{enumerate}

\noindent
The unknown  $C_{0}$ is calculated
from the knowledges of $C_K-C_{0}$ for all $K\in\Th$, since
\[0=\int_\O \ph^\CC\ dxdy = \sum_{K\in\Th}C_K|K|
  = \sum_{K\in\Th}(C_K-C_{0})|K|+C_{0}|\O|. \]
\begin{remark}
  For an efficient choice of $\V$ in (ii), we could use a structure of the mesh
  such as a hierarchy. Dividing $\O$ into subdomains, the above algorithm 
  would be easily parallelized. 
\end{remark}

\section{A locally calculable $P^3$-pressure}\label{sec:def-p3-pressure}
We have prepared the locally calculable components $\ph^\NS, \ph^\ST, \ph^\CC$
in Lemma \ref{lem:phns-error}, \ref{lem:phst-error}, \ref{lem:phcc-error}, respectively.
At last, we arrive at the following
final definition of a pressure $\ph\in\PO3\cap L_0^2(\O)$: 
\begin{equation}
  \label{def:ph}
  \ph=\ph^\NS+ \ph^\ST  + \ph^\CC -\m(\ph^\ST).
\end{equation}
\begin{theorem}\label{thm:last}
  Let $(\u,p)\in [H_0^1(\O)]^2 \times L_0^2(\O)$ satisfy \eqref{prob:conti}.
  Then, for $\ph$ defined in  \eqref{def:ph},
  we estimate
\begin{equation}\label{eq:thm-last}
   \|p-\ph\|_0 \le C h^4(|\u|_5+|p|_4),
 \end{equation}
  if $(\u,p)\in [H^5(\O)]^2\times H^4(\O)$.
\end{theorem}
\begin{proof}
  By \eqref{eq:decom-pihp} and \eqref{def:ph}, we expand
  \begin{equation}
    \label{eq:thm-last-1}
    \pihp-\ph=\left(\pihp^\NS -\ph^\NS\right) +\left(\pihp^\ST-\ph^\ST\right)+
    \left(\pihp^\CS-\ph^\CC +\m(\ph^\ST)\right).
  \end{equation}
  Since $  \m(\pihp)=\m(\pihp^\CS)+\m(\pihp^\ST)$, the last term in \eqref{eq:thm-last-1}
  makes
  \begin{equation}
    \label{eq:thm-last-2}
\pihp^\CS-\ph^\CC +\m(\ph^\ST)=\left(\pihp^\CS-\m(\pihp^\CS)-\ph^\CC \right)
  +\m(\ph^\ST-\pihp^\ST)+ \m(\pihp).
\end{equation}
We note that
  \begin{equation}
    \label{eq:thm-last-3}
    |\m(\pihp)|=|\m(\pihp-p)| \le |\O|^{-1/2} \|\pihp-p\|_0.
  \end{equation}
  Therefore, \eqref{eq:thm-last} is established from \eqref{eq:pihp-inter-error},
  \eqref{eq:thm-last-1}-\eqref{eq:thm-last-3}
  and Lemma \ref{lem:phns-error}, \ref{lem:phst-error}, \ref{lem:phcc-error},
  and Theorem \ref{th:vel-error}.
\end{proof}

\section{A summary of the method}\label{sec:sum}
\begin{description}
  \item[Step 1.] 
Calculate  $\u_h\in \Z_{h,0}$ in \eqref{def:zh0} such that
\begin{equation*}
  (\nabla \u_h, \nabla \v_h)= (\f,\v_h)\quad \mbox{ for all } \v_h\in \Z_{h,0}.
\end{equation*}

\item[Step 2.]
For each triangle $K\in\Th$, calculate $p_h^K\in \NS(K)$ in \eqref{def:nonsting}
such that
\begin{equation*}
  (p_h^K,\div\v_h)=(\f,\v_h)-(\nabla\u_h,\nabla\v_h)\quad\mbox{ for all }
    \v_h\in \BB(K) \mbox{ in } \eqref{def:BK}.
  \end{equation*}
  Then, denote
  \[ p_h^\NS = \sum_{K\in\Th} p_h^K.\]
  
\item[Step 3.]
  For each regular vertex $\V$  as in Figure \ref{fig:reg}, calculate the least square solution  $\ph^ \V\in \ST(\V)$ in \eqref{def:STV}
  for a system of $2\JJ$ equations:
\begin{equation*}\arraycolsep=1.4pt\def\arraystretch{1.6}
  \begin{array}{lll}
(p_h^\V, \div\w_h^{\t_j})&=&(\f,\w_h^{\t_j})-(\nabla\u_h, \nabla\w_h^{\t_j})-(p_h^\NS, \div\w_h^{\t_j}),\\
    (p_h^\V, \div\w_h^{\t_j^\perp})&=&(\f,\w_h^{\t_j^\perp})-(\nabla\u_h, \nabla\w_h^{\t_j^\perp})-(p_h^\NS, \div\w_h^{\t_j^\perp}),
  \end{array}
\end{equation*}
for $j=1,2,\cdots,\JJ$, 
where $\JJ$ is the number of interior edges of $\V$ and
$\w_h^{\t_j}, \w_h^{\t_j^\perp}$ are test functions defined in
\eqref{cond:wk}-\eqref{def:wht}.

\item[Step 4.]
  For each nearly singular vertex $\V$ meeting an interior edge  as in Figure \ref{fig:sing},
  calculate the least square solution  $\ph^ \V\in \ST(\V)$ for a system of $2\JJ$ equations:
\begin{equation*}
  \arraycolsep=1.4pt\def\arraystretch{1.6}
  \begin{array}{lll}
  (p_h^\V, \div\w_h^{\t_j})&=&(\f,\w_h^{\t_j})-(\nabla\u_h, \nabla\w_h^{\t_j})-(p_h^\NS, \div\w_h^{\t_j}),\\
    \jump_{\jjp} (p_h^\V) &=& -\jump_{\jjp}(\ph^\NS+p_h^{\V_j} ),\quad
  \end{array}
\end{equation*}
for $j=1,2,\cdots,\JJ$,
where $\jump_{\jjp}$ is the jump defined in \eqref{def:jumpkk1}
and $\V_1, \V_2,\cdots,\V_\JJ$ are all regular vertices sharing interior edges with $\V$.

\item[Step 5.]
For each nearly singular vertex $\V$ meeting no interior edge  as in Figure \ref{fig:dead}, calculate
$\ph^ \V\in \ST(\V)$ such that
\begin{equation*}
  \jump(\ph^\V)=-\jump(\ph^\NS+\ph^{\W_1}+\ph^{\W_2}+\ph^{\W_3}),
\end{equation*}
where $\jump$ is the jump defined in \eqref{def:jump12n}
and $\W_1, \W_2, \W_3$ are all vertices of the triangle whose intersection with $\KK(\V)$ is an edge.

\item[Step 6.]
Calculate $\ph^\CC \in \PO0\cap L_0^2(\O)$ such that,
for every two triangles $K_1, K_2$ sharing an edge,
\begin{equation*}
   \ph^\CC\big|_{K_1}- \ph^\CC\big|_{K_2}=(\f,\v_h)-(\nabla\u_h,\nabla\v_h)-(p_h^\NS,\div\v_h),
\end{equation*}
where $\v_h\in V_{h,00}$ in \eqref{def:vtx-c0} is a test function
satisfying \eqref{cond:vh-for-qhc}.

\item[Step 7.]
  Denote
 \[ \ph^{\ST} = \sum_{\V:\mbox{\rm{vertex}}} \ph^{\V}.\]
 Then, define $\ph\in\PO3\cap L_0^2(\O)$ as
\[ \ph=\ph^\NS +\ph^\ST + \ph^\CC -\oint_\O \ph^\ST\ dxdy.  \]
\end{description}  

\section{Numerical tests}
All of the numerical tests were done with the velocity $\u$ and pressure $p$ on  $\O=[0,1]^2$ such that
\[ \u=\left(s(x)s'(y), -s'(x)s(y)\right),\quad
  p=\sin(4\pi x)e^{\pi y},\quad\mbox{ where }s(t)=(t^2-t)\sin(2\pi t).  \]

\subsection{suggested method over singular meshes}
We tested the suggested method over singular meshes as in Figure \ref{fig:mesh}.
For triangulations, we formed first the meshes of uniform squares
over $\O$, then added one exactly singular vertex in every non-corner square.
For the corners, we made them singular as in Figure \ref{fig:mesh},
an example of $8\times 8\times 4$ mesh.

We calculated locally the components $\ph^\NS, \ph^\CC$ 
as well as $\ph^\V$ for all vertices $\V$ in order:
regular vertices, interior singular vertices, dead corners. 
In Figure \ref{fig:ph}, their superposition on  making $\ph$ in \eqref{def:ph} 
are depicted for the mesh in Figure \ref{fig:mesh}.

The errors in Table \ref{table} show the optimal order of convergence,
expected in Theorem \ref{th:vel-error} and \ref{thm:last}.
We used a direct linear solver in {\texttt {LAPACK}}
on solving the problems \eqref{eq:dclv-SC} for $\u_h$
in double precision.

\begin{figure}[ht]
\centering
\includegraphics[width=0.32\linewidth]{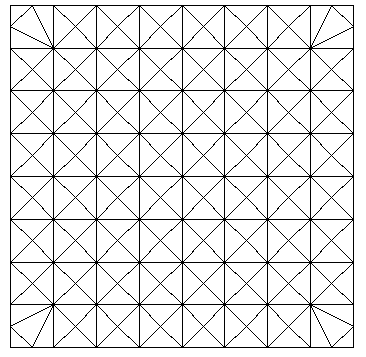}
\caption{$8\times8\times4$ singular mesh, each unit square has a singular vertex}
\label{fig:mesh}
\end{figure}

\begin{table}
  \centering
\begin{tabular}{r || c c || c c }
\hline
  mesh\hspace{3mm}
  & \hspace{1mm}$|\u-\u_h|_1$\hspace{1mm} & order\hspace{1mm}
     & \hspace{1mm}$\|p-p_h\|_0$\hspace{1mm}
  & order \\
\hline
 8 x 8 x 4 & 7.3894E-4  &  & 4.3010E-3 &   \\ 
 16 x 16 x 4 & 3.7236E-5  &4.31  & 1.8565E-4  &4.53  \\
  32 x 32 x 4 &2.2793E-6  &4.03  &1.0805E-5  & 4.10 \\
   64 x 64 x 4 &1.3859E-7 & 4.04 & 6.5962E-7  & 4.03 \\ 
\hline
\end{tabular}
\caption{\label{table}the errors for the suggested method over singular meshes as in Figure \ref{fig:mesh}}
\end{table}

\subsection{comparison with mixed FEM}
We calculated the discrete solutions
$(\u_h^{\rm{FN}}, \ph^{\rm{FN}})$, $(\u_h^{\rm{SV}}, \ph^{\rm{SV}})$
and $(\u_h, p_h)$
by the Falk-Neilan \cite{Falk2013}, Scott-Vogelius \cite{Guzman2018, Scott1985} and  suggested methods, respectively,
over regular meshes  as in Figure \ref{fig:mesh23} and
nearly singular meshes  as in  Figure \ref{fig:mesh99}.
For easy comparison,
a common  direct linear solver was used for  all involving linear systems.

As shown in Table \ref{table-2:3v} and \ref{table-99:100v},
the suggested and Falk-Neilan methods were almost same for the velocity,
since the divergence-free subspace of Falk-Neilan for $\u_h^{\rm{FN}}$
is slightly different to  $\Z_{h,0}$ for $\u_h$ in \eqref{def:zh0}
by merely a few elements for the corners.
In those tables, a little advantage was lying on the Scott-Vogelius method
as expected from its larger divergence-free subspace.

For the pressure, the suggested and Falk-Neilan methods offered more favorable errors
as in Table \ref{table-2:3p} and  \ref{table-99:100p}.
It is acceptable since they reflect some continuity of pressure.

The results in Table \ref{table-99:100p} implied that
the pressures by Scott-Vogelius were ruined over nearly singular meshes.
It is also confirmed in Figure  \ref{fig:3press},
where the pressures calculated over the $8\times8\times4$ nearly singular mesh 
are depicted.
This unstable phenomena is well known and recently turns out due to
the characteristic of the sting function on the singular vertex.
Based on that,
we could recover a stable pressure from the ruined one
by simple postprocess utilizing vertex continuity of pressure \cite{Park2020}.

It is interesting that
the results over nearly singular meshes outperformed those over regular meshes,
except the pressures by Scott-Vogelius.
The reason why is that
the largest triangles  in Figure \ref{fig:mesh99} are smaller
than those in Figure \ref{fig:mesh23}.

In all the tests, the suggested method was comparable with other two mixed finite element methods,
while it would cost less, since it used only local computations for the pressures.

\begin{figure}[ht]
  \centering
  \subfloat[$\O$]
  {\includegraphics[width=0.32\textwidth]{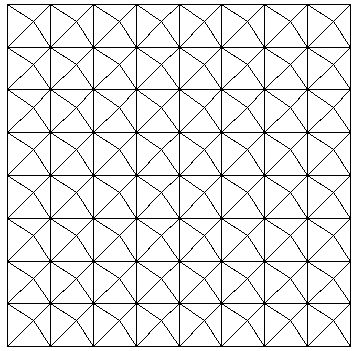}}\hspace{15mm}
 \subfloat[unit square in $\O$]{\raisebox{4ex}
  {\includegraphics[width=0.23\textwidth]{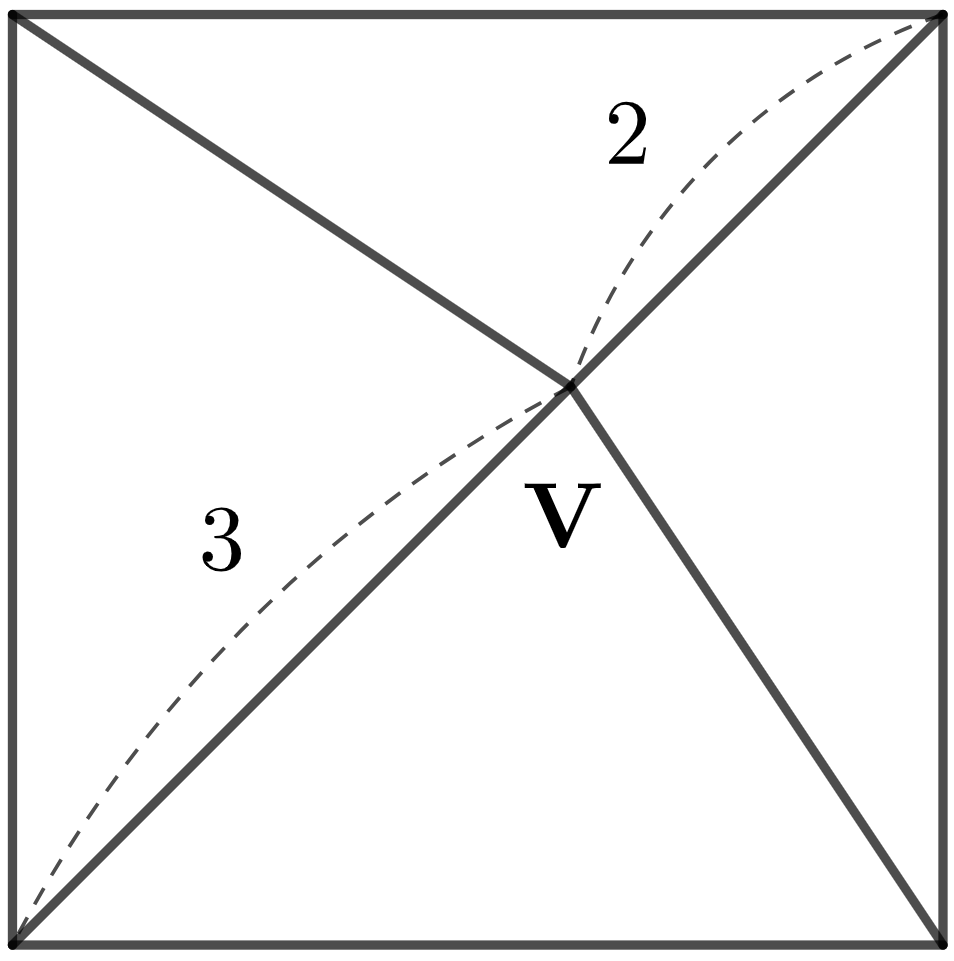}}}
  \caption{regular mesh, the vertex $\V$ divides the diagonal of positive slope
in the ratio $2:3$}
  \label{fig:mesh23}
\end{figure}

\begin{figure}[ht]
  \centering
  \subfloat[$\O$]
   {\includegraphics[width=0.32\textwidth]{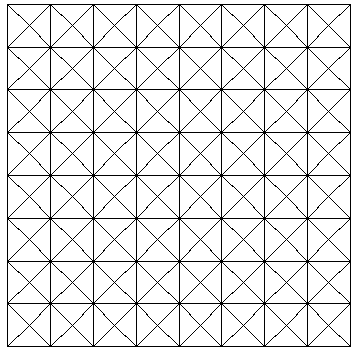}}\hspace{15mm}
  \subfloat[unit square in $\O$]{\raisebox{4ex}
  {\includegraphics[width=0.23\textwidth]{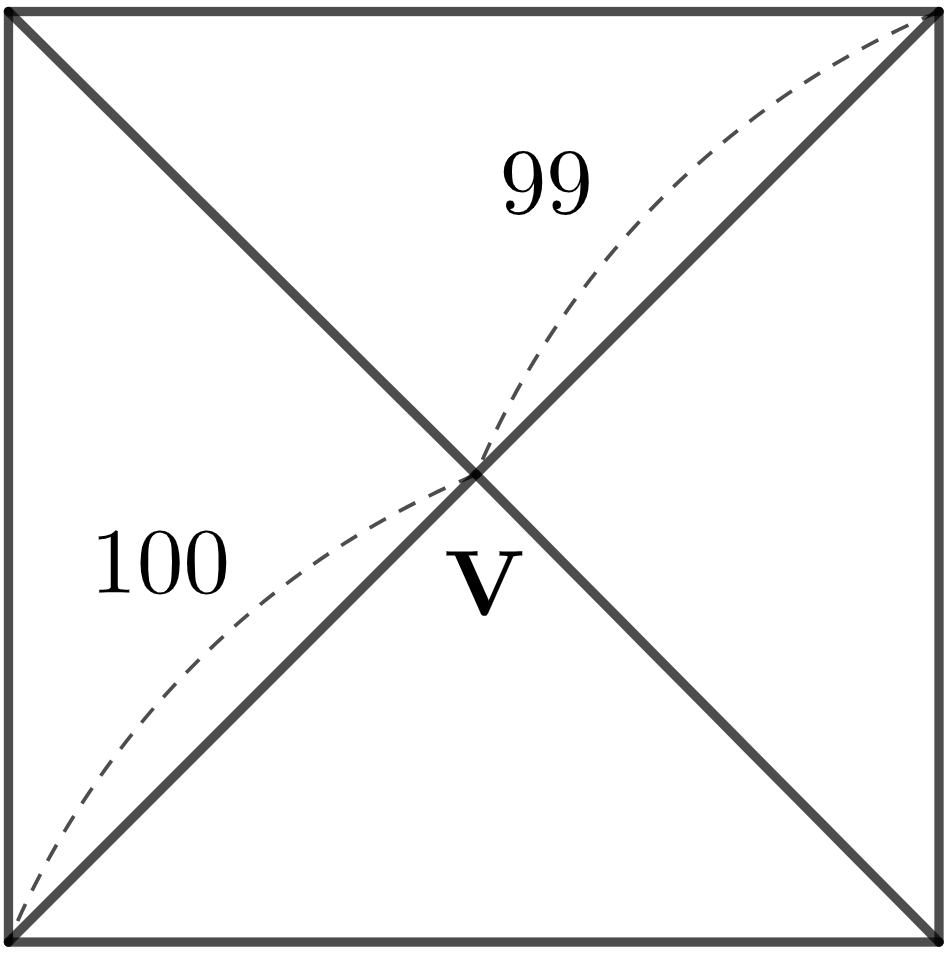}}}
  \caption{nearly singular mesh, the vertex $\V$ divides the diagonal of positive slope
in the ratio $99:100$}
  \label{fig:mesh99}
\end{figure}

\begin{table}
  \centering
\begin{tabular}{r || c c || c c || c c }
  \hline
  \rule{0pt}{10pt}
  mesh\hspace{3mm}
  & \hspace{1mm}$|\u-\u_h|_1$\hspace{1mm} & order\hspace{1mm}
  & \hspace{1mm}$|\u-\u_h^{\rm{FN}}|_1$\hspace{1mm} & order\hspace{1mm}
  & \hspace{1mm}$|\u-\u_h^{\rm{SV}}|_1$\hspace{1mm} & order \\[0.5ex] 
  \hline
  4 x 4 x 4 &1.4450E-2 & &  1.4450E-2  & &1.1706E-2 & \\
  8 x 8 x 4 & 8.5476E-4  & 4.08 & 8.5476E-4  & 4.08 &7.5823E-4 & 3.95 \\ 
  16 x 16 x 4 &5.1606E-5  &4.05 &5.1606E-5  &4.05 &4.7135E-5 & 4.01 \\
  32 x 32 x 4 &3.1882E-6 &4.02  &3.1882E-6  &4.02 &2.9271E-6 & 4.01 \\
\hline
\end{tabular}
\caption{the errors in velocity over regular meshes as in Figure \ref{fig:mesh23} }
\label{table-2:3v}
\end{table}

\begin{table}
  \centering
\begin{tabular}{r || c c || c c || c c }
  \hline
  \rule{0pt}{10pt}
  mesh\hspace{3mm}
  & \hspace{1mm}$|\u-\u_h|_1$\hspace{1mm} & order\hspace{1mm}
  & \hspace{1mm}$|\u-\u_h^{\rm{FN}}|_1$\hspace{1mm} & order\hspace{1mm}
  & \hspace{1mm}$|\u-\u_h^{\rm{SV}}|_1$\hspace{1mm} & order \\[0.5ex] 
  \hline
  4 x 4 x 4 &1.1266E-2 & & 1.1266E-2  & &8.5523E-3 & \\
  8 x 8 x 4 & 6.1513E-4  &4.19  &6.1513E-4  &4.19  &5.4485E-4 &3.97  \\ 
  16 x 16 x 4 &3.5952E-5  &4.10 &3.5952E-5  &4.10 &3.3934E-5 &4.01  \\
  32 x 32 x 4 &2.2009E-6 &4.03  &2.2009E-6  &4.03 &2.1180E-6 & 4.00 \\
\hline
\end{tabular}
\caption{the errors in velocity over nearly singular meshes as in Figure \ref{fig:mesh99}}
\label{table-99:100v}
\end{table}

\begin{table}
  \centering
\begin{tabular}{r || c c || c c || c c }
  \hline
  \rule{0pt}{10pt}
    mesh\hspace{3mm}
   & \hspace{1mm}$\|p-\ph\|_0$\hspace{1mm} & order\hspace{1mm}
  & \hspace{1mm}$\|p-\ph^{\rm{FN}}\|_0$\hspace{1mm} & order\hspace{1mm}
  & \hspace{1mm}$\|p-\ph^{\rm{SV}}\|_0$\hspace{1mm} & order \\[0.5ex] 
  \hline
  4 x 4 x 4 &6.1948E-2 & & 7.5405E-2  & &9.0916E-2 & \\
  8 x 8 x 4 &  3.1862E-3  &4.28  &3.5251E-3  &4.42  &5.3241E-3 &4.09  \\ 
  16 x 16 x 4 &1.9879E-4  &4.00 &2.1695E-4  &4.02 &3.2844E-4 & 4.02 \\
  32 x 32 x 4 &1.2413E-5 & 4.00 &1.3499E-5  &4.01 &2.0319E-5 & 4.01 \\
\hline
\end{tabular}
\caption{the errors in pressure over  regular meshes as in Figure \ref{fig:mesh23} }
\label{table-2:3p}
\end{table}

\begin{table}
  \centering
\begin{tabular}{r || c c || c c || c c }
  \hline
  \rule{0pt}{10pt}
  mesh\hspace{3mm}
  & \hspace{1mm}$\|p-\ph\|_0$\hspace{1mm} & order\hspace{1mm}
  & \hspace{1mm}$\|p-\ph^{\rm{FN}}\|_0$\hspace{1mm} & order\hspace{1mm}
  & \hspace{1mm}$\|p-\ph^{\rm{SV}}\|_0$\hspace{1mm} & order \\[0.5ex] 
  \hline
  4 x 4 x 4 &5.7969E-2 & & 7.6090E-2  & &1.1022E+0 & \\
  8 x 8 x 4 &2.7017E-3   &4.42  &3.3691E-3  & 4.50 &4.1561E-2 & 4.73 \\ 
  16 x 16 x 4 &1.6761E-4  &4.01 &2.0466E-4  &4.04 &1.3696E-3 & 4.92 \\
  32 x 32 x 4 &1.0455E-5 &4.00  &1.2636E-5  &4.02 &4.6032E-5 & 4.89 \\
\hline
\end{tabular}
\caption{the errors in pressure over nearly singular meshes as in Figure \ref{fig:mesh99}}
\label{table-99:100p}
\end{table}

\begin{figure}
  \hspace{3mm}
  \subfloat[$p=\sin(4\pi x)e^{\pi y}$]{\includegraphics[width=0.45\textwidth]{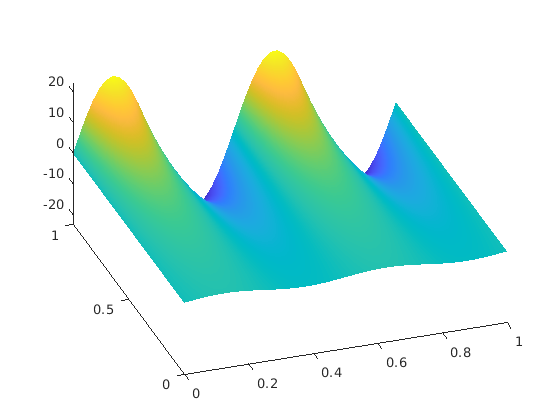}}
 \subfloat[$p_h$ by local computation]{\includegraphics[width=0.45\textwidth]{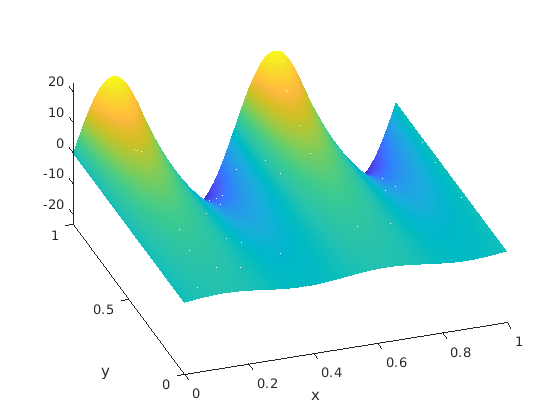}}
 
 \hspace{3mm}
 \subfloat[$\ph^{\rm{FN}}$ by Falk-Neilan]{\includegraphics[width=0.45\textwidth]{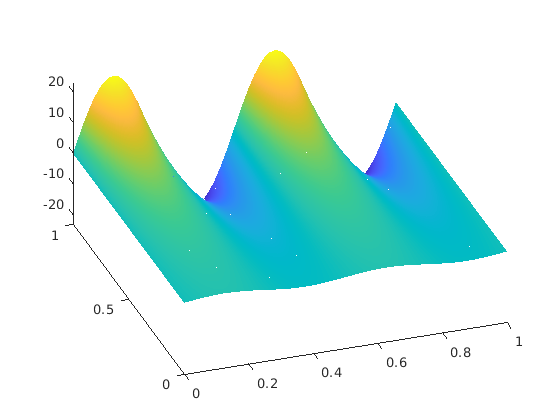}}
  \subfloat[$\ph^{\rm{SV}}$ by Scott-Vogelius]{\includegraphics[width=0.45\textwidth]{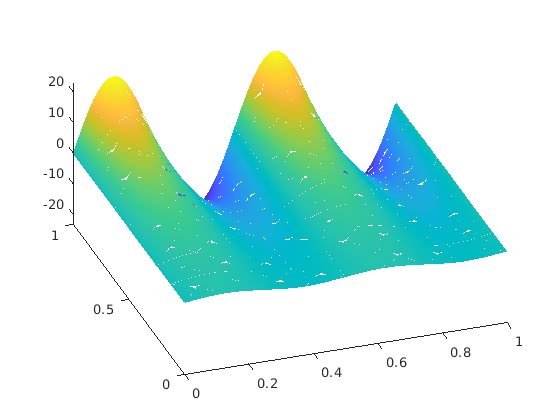}}
  \caption{the pressures calculated over the  $8\times8\times4$ nearly singular mesh in Figure \ref{fig:mesh99}}
  \label{fig:3press}
\end{figure}

\begin{figure}
  \hspace{3mm}
  \subfloat[non-sting component $\ph^\NS$]
  {\includegraphics[width=0.45\textwidth]{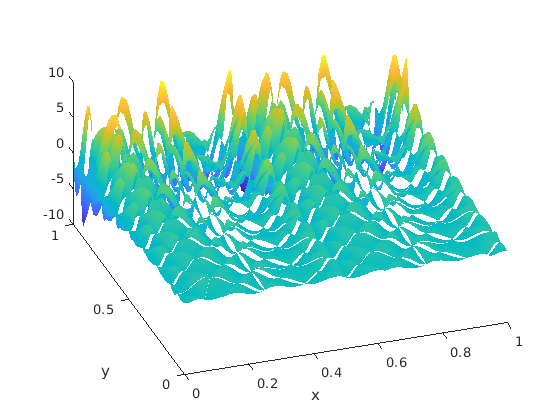}}
   \subfloat[adding piecewise constant component $\ph^\CC$ to (a)]
   {\includegraphics[width=0.45\textwidth]{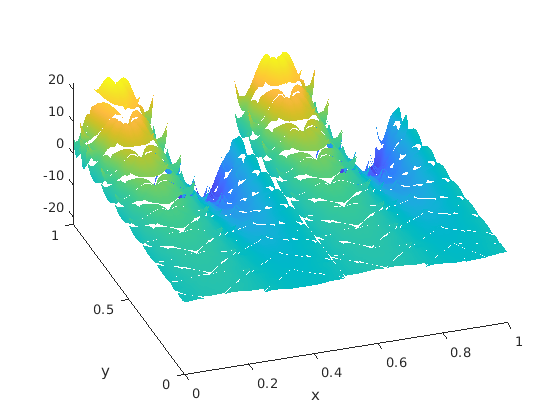}}
   
   \hspace{3mm}
   \subfloat[adding $\ph^{\V}$ for regular vertices $\V$ to (b)]
  {\includegraphics[width=0.45\textwidth]{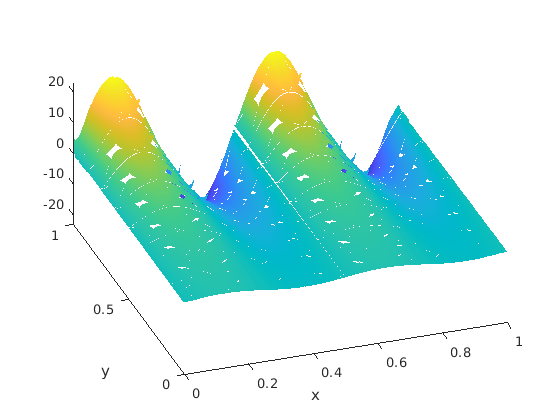}}
  \subfloat[adding $\ph^{\V}$ for interior singular vertices $\V$ to (c)] 
  {\includegraphics[width=0.45\textwidth]{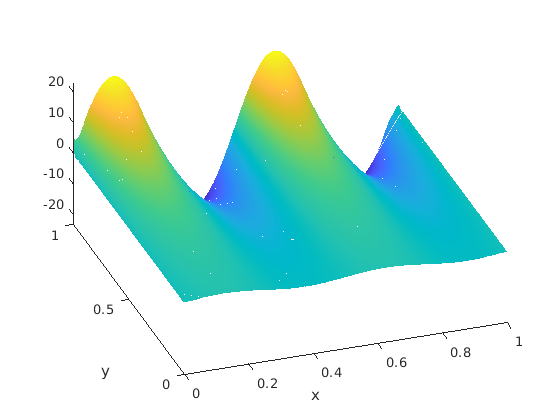}}
  
  \centering
  \hspace{3mm}
  \subfloat[$\ph=$(d)$+(\ph^{\V}$ for dead corners $\V)-\m(\ph^\ST)$]  
  {\includegraphics[width=0.45\textwidth]{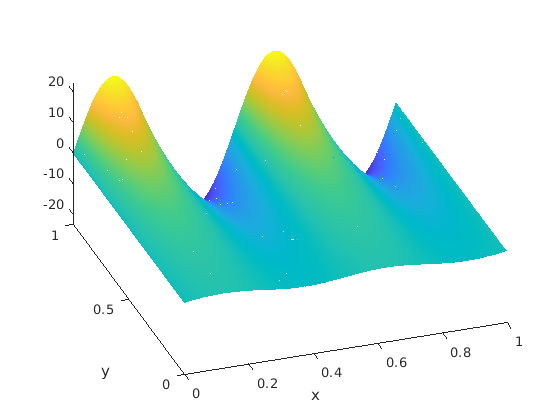}}
  \caption{superposition of the 5 components on making $\ph$
   over the mesh in Figure \ref{fig:mesh}
  }
  \label{fig:ph}
\end{figure}

 \section*{Acknowledgments}
 This work was supported by the National Research Foundation of Korea(NRF) grant
 funded by the Korea government(MSIT) (No. 2021R1F1A1055040)

 \section*{Data availability statement}
 The data that support the findings of this study are available from the corresponding author upon reasonable request.

\end{document}